\documentclass[11pt]{amsart}

\usepackage{amsmath}
\usepackage{amssymb}
\usepackage{amsthm}
\usepackage{hyperref}
\usepackage{enumerate}
\usepackage{graphicx}
\usepackage{color}
\usepackage{xcolor}
\usepackage{mathtools}
\usepackage{float}
\usepackage{tikz}
\usepackage{tikz-cd}
\usetikzlibrary{arrows}
\usepackage{bm}
\usepackage{mathrsfs}
\usepackage{soul}
\DeclareMathAlphabet\euscript{U}{eus}{m}{n}
\usepackage[all,line,
]{xy} 
\CompileMatrices

\newtheorem{theorem}{Theorem}
\newtheorem{corollary}[theorem]{Corollary}

\newtheorem{lemma}[theorem]{Lemma}
\newtheorem{proposition}[theorem]{Proposition}

\theoremstyle{definition}
\newtheorem{definition}[theorem]{Definition}
\newtheorem{remark}[theorem]{Remark}

\newtheorem{example}[theorem]{Example}

\newcommand{\per}{\operatorname{Per}}

\numberwithin{equation}{section}

\title[Embedding into near Markov shifts]{A Krieger Embedding Theorem for Near Markov Sofic Shifts}

\author[B. Marcus]{Brian Marcus}
\address[Brian Marcus]{Department of Mathematics. The University of British Columbia, Vancouver, Canada}
\email{\tt marcus@math.ubc.ca}

\author[T. Meyerovitch]{Tom Meyerovitch}
\address[Tom Meyerovitch]{Department of Mathematics. Ben Gurion University of the Negev. Be’er Sheva, 8410501, Israel}
\email{\tt mtom@bgu.ac.il}

\author[C. Wu]{Chengyu Wu}
\address[Chengyu Wu]{School of Mathematics and Statistics. Xi'an Jiaotong University, Xi'an, China}
\email{\tt chengyuw@connect.hku.hk}



\begin{document}
	
	\begin{abstract}
		Krieger's classical embedding theorem gives necessary and sufficient conditions for embedding a subshift into a mixing shift of finite type (SFT) as a proper subshift. The same result does not hold if one replaces mixing SFT by a mixing sofic shift. In this paper, we generalize Krieger's conditions to give
		necessary and 
		sufficient conditions for embedding a subshift into a mixing (in fact irreducible) near Markov sofic shift (a special conjugacy-invariant class of sofic shifts).   We also show that  if the subshift to be embedded is irreducible sofic, then the conditions are decidable. 
	\end{abstract}
	
	\maketitle
	\noindent {\em Dedicated to Benjamin Weiss on the occasion of his 85th birthday. We are forever grateful for his discovery of sofic shifts.}	
	
	\section{Introduction}
	
		
		For a subshift $X$, let $h(X)$ be the topological entropy of $X$ and  $q_n(X)$ denote the number of periodic points of least period $n$ in $X$. Krieger's celebrated embedding theorem, characterizes the  subshifts that can be embedded as a {\em proper} subshift of a mixing shift of finite type (SFT). 
		
		\begin{theorem}\label{Krieger} (W. Krieger \cite{Kr1}) Let $Z$ be a subshift and $Y$ a mixing SFT such that $h(Z) < h(Y)$. Then $Z$ embeds into $Y$ if and only if $q_n(Z) \leq q_n(Y)$ for all $n \in \mathbb N$.
		\end{theorem}
		
		This indeed gives a characterization of proper embeddings of subshifts into mixing SFTs because an embedding from a subshift $Z$ into a mixing SFT $Y$ is proper iff 
		$h(Z) < h(Y)$.
		
		Krieger's embedding theorem is one of the most central and important results in symbolic dynamics, with a wide range of applications. The necessity of the periodic point condition is trivial, as is the entropy inequality. It is the sufficiency of the conditions that is the heart of the theorem, namely the construction of an embedding. Roughly speaking, the periodic point condition enables a ``local'' embedding on a ``periodic part'' of $Z$ and the entropy condition enables a ``local'' embedding on the ``non-periodic part" of $Z$; a globally synchronizing property, using that $Y$ is a mixing SFT, enables the ``gluing together'' of the local embeddings.

		Terminology used in Theorem \ref{Krieger} above and terminology below will be defined in Section \ref{section:preliminaries}.
		
		There are known non-SFT cases where Krieger's embedding theorem holds. For example, Krieger~\cite{Kr3} showed that, for an irreducible sofic shift $Y$ whose  Fischer cover coincides with its  Krieger cover, his original conditions are necessary and sufficient for proper embeddings. However, it is not hard to see that Krieger's embedding theorem does not hold in general when the target is mixing sofic instead of being a mixing SFT. Indeed, let $Y$ be the even shift, i.e., the subshift over the alphabet $\{0,1\}$ such that $10^k1$ is allowed iff $k$ is even. Let $Z$ be an irreducible SFT containing exactly two fixed points and $q_n(Z)\leq q_n(Y)$ for all $n$ and $h(Z) < h(Y)$. If $Z$ embeds into $Y$, then $Z$ is conjugate to an irreducible sub-SFT $U$ contained in $Y$, and $U$ contains both $0^\infty$ and $1^\infty$. Let $k$ be an integer greater than the memory of $U$ as an SFT. Then, by the irreducibility of $U$, for all $k$, $10^k$ and $0^k1$ are both in the language of $U$. But this implies that $10^k1$ is allowed in $U$, contradicting the definition of $Y$ as the even shift. Thus, to see that Krieger's embedding theorem does not hold for $Y$ as the target shift, it suffices to find a subshift $Z$ satisfying the requirements above. 
		
		The existence of such a $Z$ is fairly clear.  An explicit such $Z$ is given by the edge shift defined by the finite directed graph  in Figure \ref{example_noembedding_into_sofic} below. Here, we have  $h(Z)\approx \log 1.46557$, which is less than $h(Y) = \log \frac{\sqrt{5}+1}{2}$. By the estimates in \cite[Page 351]{LM} one obtains that $q_n(Y)\leq q_n(W)$ for all $n\geq 14$, where $W$ is the golden mean shift. Since $q_n(W)=q_n(Y)$ for all $n\geq 2$, we have $q_n(Z)\leq q_n(Y)$ for all $n\geq 14$. By direct computation, one can verify that  $q_n(Z)\leq q_n(Y)$ for all $1\leq n<14$ (note that this holds trivially for $1 \le n \le 5$ since $q_1(Z) = 2 = q_1(Y)$ and for $2 \le n \le 5$, $q_n(Z) = 0$). 
		Thus, by the argument above, $Z$ does not embed into $Y$.
		
		\begin{figure}[H]
			\label{graph}
			\begin{center}
				\begin{tikzpicture}[scale=0.35]
					\draw [opacity=0] (0,0) grid (21,6);
					\node [circle, draw, thick] (a1) at (3,3) {};
					\node [circle, draw, thick] (b1) at (8,6) {};
					\node [circle, draw, thick] (c1) at (13,6) {};
					\node [circle, draw, thick] (d1) at (18,3) {};
					\node [circle, draw, thick] (e1) at (13,0) {};
					\node [circle, draw, thick] (f1) at (8,0) {};
					\draw [-stealth, black, thick] (a1) -- node[below] {} (b1);
					\draw [-stealth, black, thick] (b1) -- node[below] {} (c1);
					\draw [-stealth, black, thick] (c1) -- node[below] {} (d1);
					\draw [-stealth, black, thick] (d1) -- node[below] {} (e1);
					\draw [-stealth, black, thick] (e1) -- node[below] {} (f1);
					\draw [-stealth, black, thick] (f1) -- node[below] {} (a1);
					\draw [-stealth, black, thick] (a1) to [out=240, in=120, style={min distance=6cm}] (a1);
					\draw [-stealth, black, thick] (d1) to [out=300, in=60, style={min distance=6cm}] (d1);
				\end{tikzpicture}
			\end{center}
			\caption{The edge shift $Z$}
			\label{example_noembedding_into_sofic}
		\end{figure}
        
		This phenomenon has been known since the time of Krieger's embedding theorem; indeed, see ~\cite[Example 3.1]{B} for a more elaborate example $Y$.

		In the current paper, we are concerned with the case that the given target shift is in a conjugacy invariant class, for which Krieger's necessary conditions are not sufficient for the existence of proper embeddings. The main results of our paper, Theorem~\ref{embedding_into_near_markov} (characterization of embeddings) and Theorem~\ref{finite-deci-main-theorem} 
		(decidability of characterization), give a complete and decidable characterization of existence of proper embeddings of irreducible sofic shifts into a given target sofic shift $Y$ in this conjugacy-invariant class. This is the class of {\em near Markov shifts,} which are defined as those that can be presented as a factor of an irreducible SFT which is one-to-one off of a finite set.  These are, in some sense, the sofic shifts that are as close as possible to being SFTs, without actually being SFTs. This class includes the even shift and other irreducible sofic shifts defined by forbidding words in the union of a finite list of words and a finite list of congruence conditions. Our conditions are necessarily stronger than Krieger's original necessary conditions when the target shift is a near Markov shift that is not of finite type. 
		
		In the case when the target shift $Y$ is  the even shift, our main result, Theorem~\ref{embedding_into_near_markov}, reduces to the proposition below, which is proven at the end of Section \ref{sec:embedding_for_near_markov}.
		
		\begin{proposition} \label{embedding_even}
			Let $Z$ be a subshift. Then $Z$ embeds into the even shift $Y$ if and only if one of the following conditions hold:
			\begin{enumerate}
				\item $Z$ is  conjugate to $Y$.
				\item $h(Z)< h(Y)$, $Z$ has at most $1$ fixed point, and $q_n(Z)\leq q_n(Y)$ for all $n \ge 2$. 
				\item $h(Z) < h(Y)$, $Z$ has exactly $2$ fixed points, $q_n(Z)\leq q_n(Y)$ for all $n \ge 2$, 
				and $Z$ admits a $2$-blowup at one of the fixed points of $Z$; that is, there exist a subshift $\hat{Z}$ and a factor code $\pi: \hat{Z}\to Z$ such that for all $z\in Z$, $\vert \pi^{-1}(z)\vert =1$ if $z$ is not that fixed point, and $\pi^{-1}{(z)}$ is an orbit of length $2$ if $z$ is that fixed point.
			\end{enumerate}
		\end{proposition}
		We remark here that \cite[Theorem 8.6]{T} gave necessary and sufficient conditions for embedding a mixing sofic shift $X$ that satisfies certain special conditions into  a general  sofic shift $Y$. In particular these additional conditions hold when $X$ is a mixing SFT.
		
		We also remark that Krieger~\cite{Kr2} gave necessary and sufficient conditions for the factor problem, analogous to the embedding problem, of given mixing sofic shifts $Y$ and $Z$, with $h(Z) > h(Y)$, factoring $Z$ onto $Y$.

		The remainder of the paper is organized as follows. In Section \ref{section:preliminaries}, we review basic concepts and important results from symbolic dynamics. In Section \ref{sec:embedding_for_near_markov}, we state and prove our main result, Theorem \ref{embedding_into_near_markov}, which characterizes when a subshift can be embedded into a near Markov shift. The condition involves the existence of a so-called blowup of a subshift, whose properties are investigated in Section \ref{section:properties_of_blowups}. Finally, in Section \ref{section:finite_decidablity}, we show that given an irreducible sofic shift $Z$ and an irreducible near Markov shift $Y$, the problem of properly embedding $Z$ into $Y$ is  decidable.  We also investigate the relation between irreducible SFT covers and irreducible blowups of irreducible sofic shifts.

		\section{Preliminaries} \label{section:preliminaries}

		In this section, we review the basics of symbolic dynamics. For further information, the reader may wish to consult the textbooks~\cite{Ki},\cite{LM}. 
		
		The full $\mathcal{A}$-shift  over a finite alphabet   $\mathcal{A}$ is defined to be $\mathcal{A}^{\mathbb{Z}}$, topologized as the product space with $\mathcal{A}$ carrying the discrete topology; in particular, it is a compact metrizable space.  A {\em shift space} or {\em subshift} is a closed, equivalently compact, shift-invariant subset of a full shift; here the shift map $\sigma$ is defined by $\sigma(x) = y$ where $y_i = x_{i+1}$ for every $i$. If $X\subseteq \mathcal{A}^\mathbb{Z}$, we often refer to (the smallest possible such) $\mathcal{A}$ as the {\em alphabet of $X$}. Later in the paper when subshifts $X, Y, Z \cdots$ are introduced without explicitly specifying their alphabets, we use $\mathcal{A}_X, \mathcal{A}_Y, \mathcal{A}_Z \cdots$ to denote their respective alphabets.


			The {\em language} $\euscript{B}(X)$ of a shift space is the set of all finite words that appear in at least one element of $X$. Elements of $\euscript{B}(X)$ are sometimes called {\em admissible words for $X$}.
			A shift space is {\em irreducible} if for all $u,v \in \euscript{B}(X)$, there is a finite word $w$ such that $uwv \in \euscript{B}(X)$. A shift space is {\em (topologically) mixing} if for all $u,v \in \euscript{B}(X)$, there is an $N$ such that for all $n \ge N$, there is a finite word $w$ of length $n$ such that $uwv \in \euscript{B}(X)$.  
			
			A subshift $X \subseteq \mathcal{A}^\mathbb{Z}$ can  be equivalently defined as the set of all points in a full shift which avoid  a (finite or countable) list   of (finite) words  ${\mathcal F} \subseteq \bigcup_{n=1}^\infty \mathcal{A}^n$. In this case we write 
			$X = X_{\mathcal F}$. When $\mathcal{F}=\{10^{2n+1}1: n\geq 0\}$, then $X_{\mathcal{F}}$ is the even shift. In the case  that there is a \emph{finite} list ${\mathcal F}$ such that $X = X_{\mathcal F}$, then  $X$ is called a {\em shift of finite type} (SFT). The well-known {\em golden mean shift} is the SFT $X_\mathcal{F}$ defined by taking $\mathcal{F}=\{11\}$. It can be easily shown and is well-known that the even shift is not an SFT.

			Shift spaces are topologized in a way that allows continuous and shift-commuting maps 
			$\phi: X \to Y$ from one shift space $Z$ to a shift space $Y$ to be represented as sliding block codes, i.e.,  for all $i$, 
			$$
            \phi(x)_i  =  \Phi(x_{i-m} \ldots  x_{i + n})
            $$  
			where $m$ and $n$ are fixed integers and 
			$\Phi$ is a map from words of length $m + n +1$ to symbols. 
			Sometimes we write $\phi = \Phi_\infty$, suppressing $m$ and $n$. If $m=n=0$, then $\phi$ is called a $1$-block code. 
			
			A {\em conjugacy} is a bijective sliding block code.
			An {\em embedding} is an injective sliding block code, and
			a {\em factor code} is a surjective sliding block code. In this paper, we sometimes denote an embedding from  $Z$ to $Y$ by $Z\hookrightarrow Y.$

We say that a sliding block code $\phi: X\to Y$ is a {\em recoding} of another sliding block code $\phi': X'\to Y'$ if there exist conjugacies $\alpha: X\to X'$ and $\beta: Y\to Y'$ such that $\phi'\circ \alpha= \beta\circ \phi$. It is well-known that by passing to a higher block, every sliding block code can be recoded to be a $1$-block code.

Two points $x$ and $y$ in a subshift are called {\em right-asymptotic} (resp., {\em left-asymptotic}) if there exists $N$ such that $x_n = y_n$ for all $n \ge N$ (resp.,  $n \le N$).
			
A sliding block code $\phi: X\to Y$ is {\em right-closing} (resp. {\em  left-closing}), if  whenever $x,x'$ are left-asymptotic (resp. right-asymptotic) in $X$ and $\phi(x) = \phi(x')$, then $x = x'$. 
				
A $1$-block code $\phi=\Phi_{\infty}: X\to Y$ is {\em right-resolving} (resp. {\em left-resolving}) if whenever $ab$ and $ac$ are $2$-blocks allowed in $X$ with $\Phi(ab)=\Phi(ac)$, then $b=c$ (resp., if whenever $ba$ and $ca$ are $2$-blocks allowed in $Z$ with $\Phi(ba)=\Phi(ca)$, then $b=c$).  
			
It is well-known  \cite[Exercise 5.1.11]{LM} that right-resolving (resp. left-resolving) codes must be right-closing (resp. left-closing), and any right-closing (resp. left-closing) code can be recoded to be right-resolving (resp. left-resolving). Indeed, for every right-closing (resp. left-closing) code $\phi: X\to Y$, there is a shift space $X'$ and a conjugacy $\tau: X\to X'$ such that $\phi\circ \tau$ is right-resolving (resp. left-resolving).

			A {\em sofic shift}~\cite{W} is a subshift that is the image of an SFT via a factor code. An irreducible sofic shift $Y$ is a factor $\phi: X \to Y$ of an irreducible SFT $X$~\cite[Theorem 3.2.1]{LM}; such a $\phi$ is often called a {\em cover} of $Y$. 
			
			SFTs and sofic shifts have useful graphical representations as follows. 
			A finite directed graph $G$ defines an SFT in the following ways: one is as a {\em vertex shift}, i.e. the set of all bi-infinite vertex walks  $\widehat{X_G}$ on $G$, and the other is as an {\em edge shift}, i.e the set of all bi-infinite edge walks $X_G$ on $G$. Any SFT can be represented, up to conjugacy, as a vertex shift or an edge shift.  Similarly, if one endows the vertices or edges of $G$ with a vertex labeling or an edge labeling, one obtains a sofic shift, and any sofic shift can be represented in such a way. Throughout the paper, $(G, L)$ will be used to denote an edge-labeled graph, where $G$ is a finite directed graph and $L$ is an edge-labeling function.
			
			A nonnegative matrix $A$ is called {\em irreducible} if for each pair of indices $i, j$, there is an $n\geq 0$ such that $A^n_{ij}>0$. Let $G$ be a finite directed graph, $A$ be the adjacency matrix of $G$, and $X_G$ be the edge shift given by $G$ (we sometimes also use $X_A$ to denote $X_G$ when there is no confusion). It is well-known that $X_A$ is irreducible if $A$ is an irreducible matrix.

			A shift space $Y$ is {\em near Markov} 
			\cite[section 3]{BK} if it is the image of a  factor code $\phi$ on an irreducible SFT such that 
			$\{y \in Y: |\phi^{-1}(y)| \geq 2\}$ is finite.

			The even shift is a simple example of a near Markov shift that is not an SFT.
			
			The following lemma states that a near Markov shift has a unique minimal irreducible SFT cover.
			
			\begin{lemma} \label{near-markov-minimal-cover} Let $Y$ be a near Markov shift. Then there is an unique (up to conjugacy) irreducible SFT cover $\phi: X\to Y$ such that any other irreducible SFT cover $\pi: Z\to Y$ must factor through $\phi$, i.e., there is a factor code $\tau: Z\to X$ such that the following diagram commutes:
				\begin{equation} 
					\xymatrix{
						Z \ar[dr]_\pi \ar[rr]^\tau & & X \ar[dl]^\phi  \notag\\
						&	Y & }.
				\end{equation}
			\end{lemma}

			\begin{proof}
				This follows from \cite[Theorem 9]{BKM} and the fact that a near Markov shift is in particular a so-called almost finite type (AFT) sofic shift (a sofic shift $Y$ is AFT if it is a factor of an irreducible SFT where the factor map is one-to-one on a nonempty open set).
			\end{proof}
			
			The cover $\pi: X\to Y$ satisfying Lemma \ref{near-markov-minimal-cover}
			is called the (unique) {\em minimal cover} of $Y$.
			
			Let $\hat{Y}, Y$ be subshifts and $Y_0$ be a finite subshift of $Y$. We say that a factor map $\pi: \hat{Y}\to Y$ is a {\em blowup} of $Y$ at $Y_0$ if {$Y_0= \{y\in Y: \vert \pi^{-1}(y)\vert\geq 2\}$.} 
            With $Q:=Y_0$, $P:=\pi^{-1}(Y_0)$ and $f:=\pi\vert_P$, we say that $\pi$ is an {\em $f$-blowup} of $Y$. Sometimes we also say $\hat{Y}$ is an $f$-blowup of $Y$ when there is no confusion.
			Note that blowups are right- and left-closing.
			
			The following is an immediate consequence of the definition of blowup and 
			Lemma~\ref{near-markov-minimal-cover}.
			
			\begin{corollary}\label{near-markov-minimal-cover2} 
				The unique minimal cover of a near Markov shift $Y$ is a blowup of $Y$ and is finite-to-one. 
			\end{corollary}
			
			For a positive integer $k$, by a {\em $k$-blowup}, we mean an $f$-blowup 
			where $f:P\to Q$, $Q$ is a fixed point and $P$ is a single periodic orbit {of cardinality $k$.}
			
			
			The {\em global period} of an irreducible sofic shift $Y$ is defined to be any of the eight equivalent definitions in \cite[Proposition 5.1]{MMTW}, one of which is the unique number $p$ such that there is a decomposition $Y = \sqcup_{i=0}^{p-1} Y_i$ of $Y$ into closed sets, such that {the interiors of the $Y_i$'s are pairwise disjoint}, $\sigma(Y_i) = Y_{i+1}$ with addition mod $p$, and $\sigma^p$ is mixing on each $Y_i$. Such a decomposition is unique up to cyclic permutation and it is called the {\em canonical cyclic decomposition} of $Y$. When $Y$ is an irreducible SFT, this $p$ is the same as the greatest common {divisor} of the periods of periodic points in $Y$.
			
			Also, as defined in \cite[Definition 3.6]{MMTW}, a subshift $Y$ is said to be {\em $q$-periodic} if there is a partition $Y = \biguplus_{i=0}^{q-1} Y_i$ of $Y$ into clopen sets $Y_i$ such that $\sigma(Y_i) = Y_{i+1}$ with addition mod $q$. Such a partition is referred to as a {\em cyclic partition} of $Y$. Note that $Y$ can be $q$-periodic for more than one $q$. In particular, any subshift is $1$-periodic.
			
			An irreducible SFT with global period $p$ is $q$-periodic if and only if $q$ divides $p$. 
			Only one direction of this holds in the sofic case: if an irreducible sofic shift with global period $p$ is $q$-periodic, then $q$ divides $p$ (see Proposition \ref{q-periodic-characterization-by-disjointness}), but the converse need not hold. We refer the reader to \cite[Section 5]{MMTW} for more discussion on the global period and $p$-periodicity.
			
			Throughout the paper, we use $\per(p)$ to denote the least period of a {periodic} point $p$.
			
			We also need the concept of fiber product.
			
			Let $X, Y, Z$ be shift spaces, and $\phi_X: X\to Y$ and $\phi_Z: Z\to Y$ be sliding block codes. The {\em fiber product} of $(\phi_X, \phi_Y)$ is the triple $(W, \psi_X, \psi_Z)$ where 
			$$
			W:= \{(x,z)\in X\times Z: \phi_X(x)=\phi_Z(z)\},
			$$
			$\psi_X(x,z)=x$ and $\psi_Z(x,z)=z$.
			The maps $\psi_X$ and $\psi_Z$ are called the {\em arms}~\footnote{The maps $\phi_X$ and $\phi_Z$ are usually called the legs of the fiber product.} of the fiber product.  The following diagram gives a picture of the fiber product. 
			
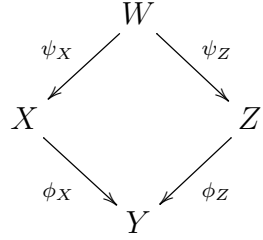
\begin{equation} 
				\xymatrix{
					& W \ar[dl]_{\psi_X} \ar[dr]^{\psi_Z}& \\
					X \ar[dr]_{\phi_X}  & & Z \ar[dl]^{\phi_Z}  \notag\\
					&	Y & }
			\end{equation}
			
			The following proposition says some properties of $\psi_X$ and $\psi_Z$ can be inherited from $\phi_Z$ and $\phi_X$, respectively.
			
			\begin{proposition}[{\cite[Proposition 8.3.3]{LM}}]\label{fibre}
				Let $X, Y, Z$ be subshifts and $(W, \psi_X, \psi_Z)$ be the fiber product of $\phi_X: X\to Y$ and $\phi_Z: Z\to Y$. For each of the following properties, if $\phi_X$ (resp. $\phi_Z$) satisfies the property, then so does $\psi_{Z}$ (resp. $\psi_X$).
				\begin{enumerate}
					\item one-to-one.
					\item onto.
					\item right-resolving.
					\item right-closing.
					\item left-resolving.
					\item left-closing.
				\end{enumerate}
			\end{proposition}

			\section{An embedding theorem for near Markov sofic shifts} \label{sec:embedding_for_near_markov}
			
			
			The main result in this section is Theorem \ref{embedding_into_near_markov} below, which gives a characterization of subshifts that can be embedded into an (irreducible) near Markov shift.

			The following technical lemma is needed in proving our main theorem. The relationships among the various maps in Lemma \ref{lem:decending_embeddings} are illustrated in Figure~\ref{cubic_figure}.
			
			
			\begin{lemma}\label{lem:decending_embeddings}
				Let $\hat Z,Z,\hat Z_0, Z_0, \hat Y, Y,\hat Y_0, Y_0$ be compact metrizable topological spaces so that $Z_0 \subseteq Z$, $Y_0 \subseteq Y$. Further, suppose that:
				\begin{itemize}
					\item $\pi_Z:\hat Z \to Z$ and $\pi_Y:\hat Y \to Y$ are continuous surjective maps.
					\item $\hat \rho: \hat Z \to \hat Y$ is a {continuous injective map}.
					\item $\hat Z_0 = \pi^{-1}_Z(Z_0)$ and $\hat Y_0 = \pi^{-1}_Y(Y_0)$.
					\item $\hat \rho^{-1}(\hat Y_0) \subseteq \hat Z_0$.
					\item $\pi_Z$ and $\pi_Y$ are injective on $\hat Z \setminus \hat Z_0$ and $\hat Y \setminus \hat Y_0$ respectively.
					\item There exists a {continuous injective map} $\rho_0:Z_0 \to Y_0$ such that  
					$\rho_0 \circ \pi_Z (\hat z)= \pi_Y \circ\hat \rho(\hat z)$ for every $\hat z \in \hat Z_0$
				\end{itemize}
				Then there exists a unique continuous injective map $\rho:Z \to Y$ that extends $\rho_0:Z_0 \to Y_0$ and satsifies $\rho \circ\pi_Z=\pi_Y \circ \hat \rho$. 
				
				Moreover, if the spaces involved are subshifts and the given maps are all shift commuting, 
				then  $\rho$ is also shift commuting.
			\end{lemma}
			
			

					
				\begin{figure}[] 
					\begin{center}
						\begin{tikzcd}[row sep=large, column sep=large]
							\hat Z_0 \arrow[rr,hook] \arrow[dd,two heads] \arrow[dr, hook] & & \hat Y_0 \arrow[dd,two heads] \arrow[dr,hook] \\
							& \hat Z \arrow[rr, crossing over,hook," \hat \rho" xshift=-10mm] \arrow[dd, crossing over,two heads,"\pi_Z" yshift=10mm] & & \hat Y \arrow[dd,two heads,"\pi_Y"] \\
							Z_0 \arrow[rr,hook,"\rho_0" xshift=10mm] \arrow[dr,hook] & & Y_0 \arrow[dr,hook] \\
							& Z \arrow[rr, hook,"\exists ! \rho"] & & Y
						\end{tikzcd}
					\end{center}
					\caption{The diagram illustrating the relationships among the various maps in  Lemma \ref{lem:decending_embeddings}.  The arrows in the diagram are of two types: {continuous injective maps}, which have one arrowhead at the head and a hook at the tail, and {continuous surjective maps}, which have two arrowheads at the head. }
					\label{cubic_figure}
				\end{figure}

				\begin{proof}
					Using surjectivity of $\pi_Z$, we can find a  section  $\psi:Z \to \hat Z$ for $\pi_Z:\hat Z \to Z$. By this we mean that $\pi_Z \circ \psi=\mathit{Id}_Z$. {Note that $\psi(Z_0)\subset \hat{Z_0}$.} Of course, we do not (and cannot) claim that  $\psi$ is continuous. However, because $\pi_Z$ is injective on the open set $\hat Z \setminus \hat Z_0$, and $\hat Z_0$ is compact, it follows that $\psi$ is continuous on $Z \setminus Z_0$.
					
					{Define $\rho: Z\to Y$ by
						$$
						\rho:= \pi_Y \circ \hat{\rho} \circ \psi.
						$$
						
						We claim that $\rho$ is a continuous injective map of $Z$ into $Y$, that moreover $\rho$ is the unique function from $Z$ to $Y$ that extends $\rho_0$ and satisfies $\rho \circ\pi_Z=\pi_Y \circ \hat \rho$.
						
						First, since $\rho_0$ satisfies  $\rho_0 \circ \pi_Z = \pi_Y \circ\hat \rho$ on $ \hat{Z_0}$ and $\psi(Z_0)\subset \hat{Z_0}$, we have $\rho_0=\pi_Y \circ \hat{\rho}\circ \psi$ on $Z_0$. Thus, $\rho$ is an extension of $\rho_0$.
					}

					
					
					
					We now show that $\rho \circ\pi_Z (\hat{z})=\pi_Y \circ \hat \rho(\hat{z})$ for all $\hat{z}\in \hat{Z}$. 
					Indeed, if $\hat{z}\in \hat{Z}_0$, then $\pi_Z(\hat{z})\in Z_0$. Thus, $\rho \circ\pi_Z (\hat{z})=\rho_0 \circ \pi_Z(\hat{z})=\pi_Y \circ \hat \rho(\hat{z})$ where the last equality follows from the last bullet in the statement of the lemma.
					If $\hat{z}\in \hat{Z}\setminus \hat{Z}_0$, then 
					$\pi_Z(\hat{z})\in Z\setminus Z_0$ and therefore
					\[
					\rho \circ \pi_Z (\hat{z}) = (\pi_Y \circ \hat \rho \circ \psi) \circ \pi_Z (\hat{z})= \pi_Y \circ \hat \rho \circ ( \psi \circ \pi_Z)(\hat{z}) = \pi_Y \circ \hat \rho(\hat{z})\]
					where the last equality follows from the fact that $\psi \circ \pi_Z= \mathit{Id}$ on $\hat{Z}\setminus \hat{Z}_0$. 
					

					

					We now check that $\rho:Z \to Y$ is injective: Suppose that $z_1,z_2 \in Z$ and $\rho(z_1)=\rho(z_2):=y \in Y$. 
					For $i=1,2$ let $\hat z_i=\psi(z_i) \in \hat Z$ and  $\hat y_i= \hat \rho(\hat z_i) \in \hat Y$. Note that $\pi_Y(\hat{y}_i)=y_i$ by the definition of $\rho$.
					
					We first consider the case  that $y \in Y_0$. Then $\hat y_1,\hat y_2 \in \pi^{-1}(Y_0) = \hat Y_0$. Since $\hat \rho^{-1}(\hat Y_0)\subseteq \hat Z_0$ it follows that $\hat z_1,\hat z_2 \in \hat Z_0$. Since $\hat Z_0=\pi_Z^{-1}(Z_0)$ and $\pi_Z(\hat z_i)=z_i$ for $i=1,2$, it follows that $z_1,z_2 \in Z_0$. Thus, because $\rho$ coincides with $\rho_0$ on $Z_0$, it follows that $y=\rho_0(z_1)=\rho_0(z_2)$. By the injectivity of $\rho_0$ it follows that $z_1=z_2$ in this case.
					
					Now consider the remaining case that $y \in Y \setminus Y_0$. Then $\hat y_1,\hat y_2 \in \pi_Y^{-1}(Y\setminus Y_0) = \hat Y \setminus  \hat Y_0$. Since $\pi_Y$ is injective on $\hat Y \setminus \hat Y_0$, it follows that $\hat y_1 = \hat y_2$. Since $\hat \rho^{-1}(\hat Y_0)\subseteq \hat Z_0$ it follows that $\hat z_1,\hat z_2 \in \hat Z \setminus \hat Z_0$. Also, because $\hat \rho$ is injective, it follows that $\hat z_1= \hat z_2$. Because $\psi$ is injective on $Z\setminus Z_0$, it follows that $z_1=z_2$ in this case as well. This completes the proof that $\rho$ is injective.

					We now prove that $\rho$ is continuous.
					Since $\psi$ is continuous on $Z \setminus Z_0$ and $Z \setminus Z_0$ is open, it follows that $\rho$ is continuous on $Z \setminus Z_0$. So it remains to prove that $\rho$ is continuous on $Z_0$. Fix $z \in Z_0$, and let $(z^{(n)})_{n=1}^\infty$ be a sequence in $Z$ such that $z =\lim_{n \to \infty}z^{(n)}$. We claim that $\rho(z^{(n)})$ is close to $\rho(z)$ when $n$ is large enough. 
					
					To see this, denote $E_z:= \pi_Z^{-1} (z)$. Since $\pi_Y\circ \hat{\rho}= \rho \circ \pi_Z$, we have 
					$$
					\pi_Y\circ \hat{\rho} (E_z) = \rho\circ \pi_Z (E_z) = \{\rho(z)\}.
					$$
					Now consider $\rho(z^{(n)})$, i.e., $\pi_Y \circ \hat{\rho} \circ \psi (z^{(n)})$.	
					Even though  $\psi({z}^{(n)})$ may not converge, we deduce from the continuity of $\pi_X$ that $\psi({z}^{(n)})$ is close to $E_z$ for large enough $n$. Now, by continuity of $\hat{\rho}$ and $\pi_Y$, we know that $\pi_Y\circ \hat{\rho}\circ \psi({z}^{(n)})$  is close to the set $\pi_Y\circ \hat{\rho} (E_z)$ for large enough $n$. 
					Since $\pi_Y\circ \hat{\rho}(E_z) = \{\rho(z)\}$, we conclude that $\rho(z^{(n)})$ is close to $\rho(z)$ for large enough $n$, proving the claim.
					
					For the uniqueness of $\rho$, note that any $\rho'$ satisfying $\rho' \circ \pi_Z = \pi_Y \circ \hat{\rho}$ must satisfy $\rho'= \rho'\circ\pi_Z\circ\psi = \pi_Y \circ \hat{\rho} \circ \psi$, which coincides with our definition of $\rho$.
					Thus, $\rho$ is unique.
					
					Finally, for the moreover part, assume that the spaces involved are subshifts, that that all the given maps are shift commuting,
					and that $\rho$ is the unique continuous injective extension of $\rho_0$ with $\rho\circ \pi_Z = \pi_Y \circ \hat{\rho}$. 
					Then 
					$$
					\sigma \rho  \sigma^{-1} \pi_Z = \sigma \rho \pi_Z \sigma^{-1}  = 
					\sigma \pi_Y \hat{\rho} \sigma^{-1} = \pi_Y \hat{\rho}
					$$
					By uniqueness of $\rho$,
					$\sigma \rho \sigma^{-1} =\rho$, 
					and so $\rho$ is shift commuting. 
				\end{proof}
				
				
				Let $Y$ be a near Markov shift. By Lemma~\ref{near-markov-minimal-cover} and Corollary~\ref{near-markov-minimal-cover2}, $Y$ has a  unique minimal cover which is a blowup, $\pi_Y: \hat{Y} \to Y$, at a finite subshift $Y_0:=\{y\in Y: \vert \pi_Y^{-1}(y)\vert\geq 2\}$ of $Y$,   with $\hat{Y}$  an irreducible SFT. Denote $\hat{Y}_0:= \pi_Y^{-1}(Y_0)$ and $\pi_0 = \pi_Y|_{\hat{Y}_0}$. So, $\pi_Y$ is a $\pi_0$-blowup. 
				\begin{theorem}
					\label{embedding_into_near_markov}
					Let $Y$ be an (irreducible) near Markov shift with $\hat{Y}, Y_0, \hat{Y}_0$  given as above.
					Let $Z$ be a subshift. Then $Z$ embeds into $Y$ iff either 
					\begin{enumerate}
						\item $Z$ is conjugate to $Y$ --or--
						\item $h(Z)<h(Y)$ and there exists a finite subshift $Z_0\subseteq Z$, a finite subshift $\hat{Z}_0$, a factor code $f_0: \hat{Z}_0 \to Z_0$ such that
						\begin{enumerate}
							\item   there are  embeddings $\rho_0: Z_0\hookrightarrow Y_0$ and $\hat{\rho_0}: \hat{Z_0}\hookrightarrow \hat{Y_0}$ s.t. the following diagram commutes:
							\begin{equation}\label{commdiag}
								\begin{xymatrix}{
										\hat{Z}_0\ar[d]_-{f_0}\ar@{^{(}->}[rr]^-{\hat{\rho_0}} & &\hat{Y}_0\ar[d]^-{\pi_0} \\
										Z_0 \ar@{^{(}->}[rr]^-{\rho_0} & & Y_0 }
								\end{xymatrix}
							\end{equation}
							\item There exists an $f_0$-blowup $\pi_Z: \hat{Z}\to Z$ at $Z_0$.
							\item $q_n({Z}\setminus {Z}_0)\leq q_n({Y}\setminus{Y_0})$ for all $n$~ \footnote{Here, we continue to use $q_n(Z\setminus Z_0)$ to denote the number of periodic points with least period $n$ in $Z\setminus Z_0$, though $Z\setminus Z_0$ may not be a subshift. The same convention applies to $q_n(Y\setminus Y_0)$.}.        
							\item $\hat{Z}$ is $p$-periodic where $p$ is the global period of $\hat{Y}$.
						\end{enumerate}
					\end{enumerate}
					Moreover, when $Z$ is irreducible, then $\hat{Z}$ above can also be taken to be irreducible.
				\end{theorem}


				\begin{proof} 
					\noindent{\bf (Necessity):}
					If $Z$ is not conjugate to $Y$ and $Z$ embeds in $Y$, then $h(Z) < h(Y)$ by entropy minimality of irreducible sofic shifts~\cite[Corollary 4.4.9]{LM}.
					
					Let $\rho: Z \to Y$ be an embedding and $Z_0 = \rho^{-1}(Y_0)$.  Let $\hat{Z}$ be the fiber product of
					$\rho$ and $\pi_Y$, with arms $\pi_Z: \hat{Z} \to Z$ and $\hat{\rho}: \hat{Z} \to \hat{Y}$, i.e.,
					$$
					\hat{Z} = \{(z,\hat{y}) \in Z \times \hat{Y}: \rho(z) = \pi_Y(\hat{y})\}, ~\pi_Z(z, \hat{y}) = z, ~\hat{\rho}(z, \hat{y}) = \hat{y}.
					$$
					Then, the following diagram commutes:
					\begin{equation}\label{one}
						\begin{xymatrix}{
								\hat{Z}\ar[d]_-{\pi_Z}\ar@{^{(}->}[rr]^-{\hat{\rho}} & &\hat{Y}\ar[d]^-{\pi_Y} \\
								Z \ar@{^{(}->}[rr]^-{\rho} & & Y .}
						\end{xymatrix}
					\end{equation}

					Since $\rho$ is an embedding, by Proposition \ref{fibre}, $\hat{\rho}$ is an embedding.
					
					Let $$
					\rho_0 = \rho|_{Z_0}, \hat{Z}_0 = \pi_Z^{-1} (Z_0), \hat{\rho_0}= \hat{\rho}|_{\hat{Z}_0}.
					$$
					Then, we immediately have $\rho_0$ is an embedding from $Z_0$ into $Y_0$. Note that $\hat{\rho_0}$ is an embedding on $\hat{Z_0}$. We need to show it maps $\hat{Z_0}$ into $\hat{Y_0}$.
					Since $Z_0 = \rho^{-1}(Y_0)$, we derive from diagram (\ref{one}) that 
					{\small\begin{equation} \label{tilde-X0-Y0}
						\hat{Z}_0 = \pi_Z^{-1} (Z_0) = \pi_Z^{-1}\circ\rho^{-1}(Y_0) = (\rho \circ \pi_Z)^{-1}(Y_0) = (\pi_Y \circ \hat{\rho})^{-1}(Y_0)
						= \hat{\rho}^{-1}(\hat{Y}_0).
					\end{equation}}
					Therefore, $\hat{\rho}_0$ is an embedding from $\hat{Z}_0$ into $\hat{Y}_0$. 
					Let $f_0:= \pi_Z|_{\hat{Z}_0}$. The diagram (\ref{commdiag}) follows from the diagram (\ref{one}). Thus, (a) holds.
					
					For part (b), note that 
					$\pi_Z$ is an $f_0$-blowup because 
					$\pi_Y$ is a $\pi_0$-blowup. 

					

					Since $\pi_Z$ and $\pi_Y$ are blowups at $Z_0$ and $Y_0$ respectively, we have 
					$q_n(\hat {Z}\setminus \hat{Z}_0)=q_n(Z\setminus Z_0)$ and $q_n(\hat{Y}\setminus \hat{Y}_0)=q_n(Y\setminus Y_0)$. Then
					Part (c) follows from Equation (\ref{tilde-X0-Y0}) and the fact that $\hat{\rho}$ is an  embedding that maps $\hat{Z}\setminus \hat{Z}_0$ into $\hat{Y}\setminus  \hat{Y}_0$.
					
					Part (d) follows from the fact that  $\hat{\rho}$ is an  embedding and the necessity of the $p$-periodic condition for embedding into irreducible SFTs  \cite[Proposition 3.7]{MMTW}. 
					
					\medskip
					
					\noindent{\bf (Sufficiency):}		
					Assume $h(Z) < h(Y)$.
					Since $Y$ is near Markov, $\pi_Y$ is finite-to-one and  $h(\hat{Y}) = h(Y)$. 	
					
					By (a), since $Y_0$ and $\hat{Y}_0$ are finite, so are 	$Z_0$ and $\hat{Z}_0$.  Thus, by (b), $h(\hat{Z}) = h(Z)$. It follows that
					\begin{equation}
						\label{two}
						h(\hat{Z}) < h(\hat{Y}).
					\end{equation}
					By (b), $\pi_Z\vert_{\hat{Z}\setminus \hat{Z}_0}: \hat{Z}\setminus \hat{Z}_0\to Z\setminus Z_0$ is one-to-one.  By assumption, $\pi_Y\vert_{\hat{Y}\setminus \hat{Y}_0}: \hat{Y}\setminus \hat{Y}_0\to Y\setminus Y_0$ is one-to-one. Thus,  (c) implies 
					\begin{equation}\label{q_n-tildes}
						q_n(\hat{Z}\setminus \hat{Z}_0) \leq q_n(\hat{Y}\setminus \hat{Y}_0) \quad \mbox{for all} ~ n
					\end{equation}
					By (a), 
					\begin{equation}\label{q_n-tildes0}
						q_n(\hat{Z}_0) \leq q_n(\hat{Y}_0)
						\quad \mbox{for all} ~ n
					\end{equation}
Now, Equations (\ref{two}), (\ref{q_n-tildes}), and 
					(\ref{q_n-tildes0}), together with condition (d), implies, by the embedding theorem into irreducible SFTs~\cite[Proposition 3.7]{MMTW}, that there is an embedding $\hat{\rho}$ of $\hat{Z}$ into $\hat{Y}$. Moreover, by a simple modification using Equation (\ref{q_n-tildes0}), we may assume that $\hat{\rho}$ extends any embedding of  $\hat{Z}_0$ into $\hat{Y}_0$, in particular one that is consistent with any embedding of $Z_0$ in $Y_0$ in the sense of diagram~(\ref{commdiag}).
					
Now, apply Lemma \ref{lem:decending_embeddings} to obtain an embedding $\rho$ of $Z$ into $Y$ that extends any embedding of $Z_0$ into $Y_0$.
					

Finally, to prove the moreover part of the theorem, assume $Z$ is irreducible, $h(Z) < h(Y)$ and $Z$ embeds into $Y$ by an embedding $\rho$. By the necessity part of the proof we may assume that (a)-(d) hold with $\hat{Z}$ being the fiber product of  $\rho$ and $\pi_Y$. 

If $Z$ is an infinite subshift, then $\rho(Z)\nsubseteq Y_0$. Since $Z$ is irreducible, there is a point $z\in Z$ whose forward orbit is dense in $Z$. It follows that the forward orbit of $\rho(z)$ is dense in $\rho(Z)$. Since $\rho(Z)\nsubseteq Y_0$, $\rho(z)$ cannot belong to $Y_0$ and therefore $\rho(z)$ has a unique $\pi_Y$-preimage $\hat{y}\in \hat{Y}$. Then $u:= (z, \hat{y})$ is in $\hat{Z}$. Let $Z'\subseteq \hat{Z}$ be the $\omega$-limit set of $u$, i.e., the set of limit points of the forward orbit of $u$. Then, $u\in Z'$, i.e., $u$ belongs to its own $\omega$-limit set $Z'$. Thus, $Z'$ is irreducible.

Since the forward orbit of $z$ is dense in $Z$, we have $\pi_Z(Z')=Z$. 
					
Since $\hat{Z}$ is an $f$-blowup of $Z$, $Z'$ is an $f'$-blowup of $Z$ where $f' = f|_{Z'}$. Since $Z' \subseteq \hat{Z}$ and $\hat{Z}$ is $p$-periodic, so is $Z'$. 
The moreover part then follows by replacing $\hat{Z}$ with $Z'$.

If $Z$ is a finite subshift, since $Z$ is assumed to be irreducible, $Z$ must be a single orbit. In this case, $\hat{Z}$ can be taken to be a single orbit, which is in particular irreducible.                
				\end{proof}
				
				We conclude this section with a proof of Proposition \ref{embedding_even}, which is a statement of our main result Theorem~\ref{embedding_into_near_markov} in the special case where the target shift space $Y$ is the even shift.

				\begin{proof}[{\bf Proof of Proposition \ref{embedding_even}}]
					In this proposition, the unique minimal cover of $Y$ is $\pi: \hat{Y} \to Y$, where $\hat{Y}$ is the golden mean shift and $\pi$ is the 2-block factor code $\pi = \Pi_\infty$ where $\Pi(10) = 0, \Pi(01) = 0, \Pi(00) = 1$,
					$Y_0$ is the single fixed point $\{0\}^\infty$ and $\hat{Y_0}$ is the orbit of $(01)^\infty$. 
					We now reduce the statement of Proposition~\ref{embedding_even} to the statement of Theorem~\ref{embedding_into_near_markov}.
					
					Case (1) of the proposition is the same as Case (1) of Theorem~\ref{embedding_into_near_markov}.

					Since $\hat{Y}$ is mixing, its global period is $p=1$ and so  condition 2d of Theorem \ref{embedding_into_near_markov} is vacuous since every subshift is $1$-periodic. 
					
					Since $Y$ has $2$ fixed points, any $Z$ that embeds in $Y$ must clearly have at most 
					$2$ fixed points.
					
					If $Z$ has at most one fixed point, then we are in 
					Case (2) of the Proposition. We can choose $Z_0 = \emptyset$ in the statement of Theorem~\ref{embedding_into_near_markov}. So, in this case, conditions 2(a), 2(b) of that theorem are vacuous. And then condition 2(c) boils down to $q_n(Z) \le q_n(Y)$ for all $n \ge 2$, which we are assuming in Case (2) of the proposition, and 
					$$
					q_1(Z\setminus Z_0)  = q_1(Z) \le 1 = q_1(Y\setminus Y_0).
					$$
					
					If $Z$ has exactly two fixed points, then we are in 
					Case (3) of the proposition. In this case $|Z_0| = 1$, so $Z_0$ is a fixed point. Since $\hat{Y_0}$ is a single orbit of length $2$, so is $\hat{Z_0}$.  So, conditions 2(a), 2(b) of part (2) of Theorem \ref{embedding_into_near_markov} amount to a 2-blowup at $Z_0$, which is exactly what we are assuming in Case (3) of the proposition.
					
					And then condition 2(c) boils down to $q_n(Z) \le q_n(Y)$ for all $n \ge 2$, which we are assuming in Case (3) of the proposition, and 
					$$
					q_1(Z\setminus Z_0) = 1 = q_1(Y\setminus Y_0).
					$$
					
					Finally, if $Z$ has at least $3$ fixed points then obviously it cannot embed into $Y$, which has only $2$ fixed points.
				\end{proof}
				
				\section{Properties of blowups} \label{section:properties_of_blowups}

				The following is a characterization of subshifts that admit a blowup of a given type.  It will be useful later for proving the decidability of the conditions in Theorem \ref{embedding_into_near_markov} in the case that $Z$ is irreducible sofic.


				
				\begin{proposition} \label{characterization-blowup}  Let $Z$ be a subshift, and $P$ and $Q$ be finite subshifts with $Q\subseteq Z$.   Let $f: P\to Q$ be a factor code.
					For each $\bar{q}\in Q$, 
					let 
					\begin{align*}
						U_m(\bar{q})&:=     \{z\in Z: z_{-1} \ne {\bar q}_{-1},  z_{[0, m-1]} = \bar{q}_{[0, m-1]} \}  \\
						V_m(\bar{q})&:= \{z\in Z: z_{1} \ne {\bar q}_{1}, z_{[-m+1,0]} =  \bar{q}_{[-m+1,0]}  \}.
					\end{align*}
					Then, there is an $f$-blowup $\pi: \hat{Z} \to Z$ of $Z$  iff 
					for each $\bar{q}\in Q$, there exist $N\in \mathbb{N}^+$ and continuous functions $c_{\bar{q},N}^+: U_{N}(\bar{q}) \to f^{-1}(\bar{q})$ and $c_{\bar{q},N}^-:V_N(\bar{q}) \to  f^{-1}(\bar{q})$ s.t. whenever  $u_{(-\infty,-1]}\bar{q}_{[0,n]}v_{[1,\infty)}\in Z$ for some $n\geq {N-1}$, some $u\in U_N(\bar{q})$ and $v\in V_N(\sigma^{n}(\bar{q}))$, 
					then  
					\begin{align} \label{congruence_condition}
						\sigma^n(c^+_{\bar{q},N}(u)) =  c^-_{\sigma^{n}(\bar{q}),N}(v).
					\end{align}
                     It can further be arranged so that the value of $c^+_{\bar{q}, N}(z)$ depends only on $\bar{q}$ and $z_{(-\infty, N]}$, and the value of $c^-_{\bar{q}, N}(z)$ depends only on $\bar{q}$ and $z_{[-N, +\infty)}$.
					
					\noindent {Moreover, if $\hat{Z}$ is irreducible, then  for any 
						$\bar{p},\bar{p}'\in P$,
						there exists  
						$u\in U_N(f(\bar{p}))$ and  $v\in V_N(f(\bar{p}'))$ such that $c^+_{f(\bar{p}),N}(u)=\bar{p}$ and $c^-_{f(\bar{p}'), N}(v)=\bar{p}'$.}
				\end{proposition}

					%

			\begin{proof}
				\noindent{\bf (Necessity:) } 
				Suppose $\pi: \hat{Z}\to Z$ is an $f$-blowup of $Z$ at $Q$. By passing to a higher block version of $Z$, we may assume that  
				the symbols of any two distinct orbits in  $Q$ are disjoint. We may also assume that 
				the symbols of any two distinct orbits in  $P$ are disjoint.

				
				First for each $\bar{q}\in Q$ and each $M\in \mathbb{N}^+$, by our assumption on symbols in  elements in $Q$, $U_M(\bar{q})$ does not contain any $\bar{q}\in Q$. Moreover, both $U_M(\bar{q})$ and $\pi^{-1}(U_M(\bar{q}))$ are compact. Thus, $\pi\vert_{\pi^{-1}(U_M(\bar{q}))}$ is a continuous bijection between compact sets and it is therefore a homeomorphism.
				


				{Now since $\pi$ is continuous, for any $\bar{q}\in Q$ and any point in $Z\setminus Q$ that is right-asymptotic  to $\bar{q}$, it has a unique preimage that is right-asymptotic to an element in $f^{-1}(\bar{q})$. It follows that if a point in $Z\setminus Q$ agrees with $\bar{q}$ on a long block starting at the $0$-th coordinate, then its unique preimage will agree with an element in $f^{-1}(\bar{q})$ on a (possibly slightly smaller) long block starting at some positive coordinate.
					In particular,} there is an $N_1(\bar{q})\in \mathbb{N}^+$ such that for all $m\geq 2N_1(\bar{q})$ and all $u\in U_m(\bar{q})$, $\pi^{-1}(u)_{[N_1(\bar{q}), m-N_1(\bar{q})]}$ is a subblock of some element in $\pi^{-1}(\bar{q})$ (which is a subset of $P$).
				Similarly, for each $\bar{q}\in Q$, there is an $N_2(\bar{q})\in \mathbb{N}^+$ such that for all $m\geq 2N_2(\bar{q})$ and all $v\in V_m(\bar{q})$, $\pi^{-1}(v)_{[-m+N_2(\bar{q}), -N_2(\bar{q})]}$ is a subblock of some element in $\pi^{-1}(\bar{q})$.
				
				Let $N_0:= \max_{\bar{q}\in Q}\max\{N_1(\bar{q}), N_2(\bar{q})\}$ and define 
				$$
				N:= 2N_0+2\max\{\per({p}): p\in P\}.
				$$
				Then for all $\bar{q}\in Q$ all $m\geq N$ and all $u\in U_m(\bar{q})$ and $v\in V_m(\bar{q})$, $\pi^{-1}(u)_{[N_0, m-N_0]}$ and $\pi^{-1}(v)_{[-m+N_0, -N_0]}$ are both subblocks of elements in $\pi^{-1}(\bar{q})$; moreover, since the lengths of $\pi^{-1}(u)_{[N_0, m-N_0]}$ and $\pi^{-1}(v)_{[-m+N_0, -N_0]}$ are greater than or equal to twice the least period of any $\bar{p}\in P$, each of $\pi^{-1}(u)_{[N_0, m-N_0]}$ and $\pi^{-1}(v)_{[-m+N_0, -N_0]}$ determine a unique periodic point in $P$ in the sense that  there exists a unique $\bar{p}\in \pi^{-1}(\bar{q})$ such that $\pi^{-1}(u)_{[N_0, m-N_0]}= \bar{p}_{[N_0, m-N_0]}$, and there exists a unique $\bar{p}'\in \pi^{-1}(\bar{q})$ such that  $\pi^{-1}(v)_{[-m+N_0, -N_0]}= \bar{p}'_{[-m+N_0, -N_0]}$.
				
				Now define
				\begin{align} \label{d+-d-_defnintino}
					\begin{dcases}
						c_{\bar{q},N}^+(u):=\bar{p}  \qquad \mbox{for all } u\in U_N(\bar{q}) \\
						c_{\bar{q},N}^-(v):=\bar{p}'  \qquad \mbox{for all } v\in V_N(\bar{q}).
					\end{dcases}
				\end{align}
				Then, $c^+_{\bar{q},N}$ and $c^-_{\bar{q},N}$ are continuous because $\pi^{-1}$ is continuous on each $U_N(\bar{q})$ and $V_N(\bar{q})$.
				
				Now suppose $u_{(-\infty,-1]}\bar{q}_{[0,n]}v_{[1,\infty)}\in Z$ for some $n\geq N-1$, some $u\in U_N(\bar{q})$ and $v\in V_N(\sigma^n(\bar{q}))$.
				Let $x:=\pi^{-1}(u_{(-\infty,-1]}\bar{q}_{[0,n]}v_{[1,\infty)})$. Since $u\in U_N(\bar{q})$, we infer from the definition of $c^+_{\bar{q},N}$ that $c^+_{\bar{q},N}(u)$ is the unique $\bar{p}\in \pi^{-1}(\bar{q})$ such that 
				\begin{equation} \label{bar-p-phase}
					\bar{p}_{[N_0, n-N_0]}=x_{[N_0, n-N_0]}.
				\end{equation}
				Similarly, since $v\in V_N(\sigma^n(\bar{q}))$, we infer from the definition of $c^-_{\sigma^n(\bar{q}),N}$ that $c^-_{\sigma^n(\bar{q}),N}(v)$ is the unique $\bar{p}'\in \pi^{-1}(\bar{q})$ such that 
				\begin{equation} \label{bar-p'-phase}
					\sigma^{-n}(\bar{p}')_{[N_0, n-N_0]}=x_{[N_0, n-N_0]}.
				\end{equation}
				Since we assume that the symbols in a least period of $\bar{p}$ and $\bar{p}'$ are all distinct, we conclude from (\ref{bar-p-phase}) and (\ref{bar-p'-phase}) that $\bar{p}=\sigma^{-n}(\bar{p}') $, which is equivalent to 
				$$
				c^+_{\bar{q},N}(u)= \sigma^{-n}(c^-_{\sigma^n(\bar{q}),N}(v)),
				$$
				as desired.
				


			\medskip

			\noindent {\bf (Sufficiency:)}
			Let $\mathcal{A}_Z$ and $\mathcal{A}_P$ denote the alphabet of symbols that occur in points of $Z$ and $P$ respectively. We may assume that $\mathcal{A}_Z$ and $\mathcal{A}_P$ are disjoint, distinct orbits in $P$ have distinct symbols, and that $f: P\to Q$ is induced by a one-block map $F$ on $\mathcal{A}_P$.

			Let $\mathcal{B} = \mathcal{A}_Z \cup \mathcal{A}_P$. We define an injective {shift-invariant} map $g:Z \setminus Q\to \mathcal{B}^{\mathbb Z}$ as follows: $g(z)_i = z_i$ except when either of the following holds:
			\begin{enumerate}
				\item[(A)] there exists $m<i$ and $\bar{q}\in Q$ such that $z_{m} \ne {\bar q}_{m}$ and $z_{[m+1,K_1]}={\bar q}_{[m+1,K_1]}$ for some $K_1\geq \max\{i, N+m+1\}$;
				\item[(B)] there exists $M>i$ and $\bar{q}'\in Q$ such that  $z_{M} \ne {\bar q}'_{M}$ and  $z_{[K_2, M-1]}=\bar{q}'_{[K_2, M-1]}$ for some $K_2\leq \min\{i, M-N-1\}$.
			\end{enumerate}
			
			\noindent If (A) holds, 
			then $\sigma^{m+1}(z) \in U_N(\sigma^{m+1}(\bar{q}))$ and we define 
			\begin{equation} \label{25-02-02-a}
				g(z)_i = \sigma^{-(m+1)}(c^+_{\sigma^{m+1}(\bar{q}),N}(\sigma^{m+1} z))_i ;
			\end{equation}
			if (B) holds,  then $\sigma^{M-1}(z) \in V_N(\sigma^{M-1}(\bar{q}'))$ and 
			we define 
			\begin{equation} \label{25-02-02-b}
				g(z)_i = \sigma^{-(M-1)}(c^-_{\sigma^{M-1}(\bar{q}'),N}(\sigma^{M-1}z))_i.
			\end{equation}
			We claim that if both (A) and (B) are satisfied, 
			then Equation (\ref{25-02-02-a}) and (\ref{25-02-02-b}) coincide. 
			This is because if both (A) and (B) hold, then, since symbols of any two orbits in $Q$ are distinct, we must have $\bar{q}=\bar{q}'$ and therefore $z_{[m+1, M-1]}=\bar{q}_{[m+1, M-1]}$. Moreover, $\sigma^{m+1}(z)$ is of the form 
			\begin{align*}
				&u_{(-\infty, -1]}\hat{q}_{[0, M-m-2]} v_{[1,+\infty)} \\
				\mbox{where} \qquad &u=\sigma^{m+1}(z), \quad  v=\sigma^{M-1}(z) \quad \mbox{and} \quad \hat{q}=\sigma^{m+1}(\bar{q}).
			\end{align*}
			Thus, Equation (\ref{congruence_condition}) implies
			$$
			\sigma^{M-m-2}(c^+_{\hat{q}, N}(\sigma^{m+1}z))=c^-_{\sigma^{M-m-2}(\hat{q}),N}(\sigma^{M-1}z),
			$$
			which gives the equivalence between (\ref{25-02-02-a}) and (\ref{25-02-02-b}).


			Note that $g$ is one-to-one because if $w =g(z) \in \textnormal{Image}(g)$, then for all $i$ s.t. $w_i \in \mathcal{A}_Z$, $z_i = w_i$ and for all $i$ s.t. $w_i \in \mathcal{A}_P$, $z_i = F(w_i)$. 
			

			Let $$\hat{Z} = g(Z \setminus Q) \cup P.$$ 
			We claim that $\hat{Z}$ is a subshift. To see this, first observe that $\hat{Z}$ is shift-invariant by construction.  It remains to show that $\hat{Z}$ is closed in the full shift $\mathcal{B}^{\mathbb Z}$. 
			
			Let	$w^{(n)}$ be a sequence in $\hat{Z}$ that converges to some $w' \in \mathcal{B}^{\mathbb Z}$. We show $w' \in \hat{Z}$.
			
			If $w' \in P$, we are done. If not, 
			we may assume that 
			each $w^{(n)} \in \textnormal{Image}(g)$. For each $n$, let $z^{(n)} = g^{-1}(w^{(n)}) \in Z$. By compactness, $z^{(n)}$ has a subsequence which converges to some $z' \in Z$.  By reindexing, we may assume that $z^{(n)}$ converges to $z'$ and $w^{(n)}$ still converges to $w'$.
			
			We claim that $z' \notin Q$. To see this, suppose to the contrary that  $z' \in Q$. By the definition of $g$, for any positive integer $n_1$, there exists a positive integer $n_2$ such that for all $n$,  if $z^{(n)}_{[-n_2,n_2]} = \bar{q}_{[-n_2, n_2]}$ for some $\bar{q}\in Q$, then  $w^{(n)}_{[-n_1,n_1]}$ is a subblock of some point in $P$. It follows that $w' \in P$, a contradiction.

			Now the reader can easily check that $g$ is continuous on its domain $Z \setminus Q$.  Thus, $w' = g(z') \in \hat{Z}$, as desired. 
			
			Finally, let $\phi: \hat{Z} \rightarrow Z$ be defined by:
			$$
			\phi|_{ g(Z \setminus Q)   } = g^{-1} \quad \mbox{ and } \quad \phi(\bar{p})= f(\bar{p}) \quad \mbox{for all } \bar{p}\in P.
			$$
			Then $\phi = \Phi_\infty$ is the 1-block code defined on $\hat{Z}$ by 
			\begin{align} \label{1-block-standard-form}
				\Phi(b)= 
				\begin{cases}
					b  \qquad &\mbox{if}\quad  b \in \mathcal{A}_Z \\
					F(b)  \qquad &\mbox{if}\quad  b\in \mathcal{A}_P.
				\end{cases}
			\end{align}
			Thus $\phi$ is a (1-block) sliding block code from    $\hat{Z}$ onto  $Z$, each $z \in Z \setminus Q$ has a unique $\phi$-preimage and $\phi^{-1}(Q) = P$, thereby proving the sufficiency.

            By Equation (\ref{d+-d-_defnintino}), the value of $c^+_{\bar{q}, N}(z)$ depends only on $\bar{q}$ and $z_{(-\infty, N]}$, and the value of $c^-_{\bar{q}, N}(z)$ depends only on $\bar{q}$ and $z_{[-N, +\infty)}$.
			
			For the moreover part of the proposition, suppose $\hat{Z}$ is irreducible. Then for each $\bar{p}\in P$, there exists $w\in \hat{Z}\setminus P$ such that 
			$$
			w_{[N_0, m-N_0]}=\bar{p}_{[N_0, m-N_0]} \quad \mbox{and} \quad \pi(w)\in U_m(f(\bar{p}))
			$$
			for some $m\geq N$.
			Thus, by the definition of $c^+_{f(\bar{p}), N}$ in Equation (\ref{d+-d-_defnintino}), we have $c^+_{f(\bar{p}), N}(u)=\bar{p}$. The proof for $c^-$ is similar.
		\end{proof}


		Note that in Proposition \ref{characterization-blowup},  by recoding $\hat{Z}, Z, P, Q$ by passing to higher-block codes,  we can assume that \begin{enumerate}
			\item[(I)]\label{item:cond_I}  $f: P\to Q$ is a one-block code.
			\item[(II)]\label{item:cond_II}  the projections $z \mapsto z_0$ and $\hat z \to \hat z_0$ from $Q$ to $\mathcal{A}_Q$ and from $P$ to $\mathcal{A}_P$ are both bijective.
			\item[(III)]\label{item:cond_III} $\euscript{B}(Q)=\euscript{B}(Z) \cap \cup_{n} (\mathcal{A}_Q)^n$.
		\end{enumerate}
		In this case, by using the uniform continuity of $c^+_{N,\bar{q}}$ and $c^-_{N,\bar{q}}$ and recoding by higher-blocks, we can assume that $N=1$ and $c^+_{N,\bar{q}}$ and $c^-_{N,\bar{q}}$ only depend on blocks of length $2$. Then, Proposition \ref{characterization-blowup} can be stated in the following simpler form.

		\begin{corollary}\label{simpler-chara-blowup}
			Let $Z$ be a subshift, $P,Q$ be finite subshifts with $Q\subseteq Z$, and $f:P \to Q$ be a factor code. Let 
			\begin{align*}
				U^+&:=\{(a, q)\in(\mathcal{A}_{Z} \setminus \mathcal{A}_Q) \times (\mathcal{A}_{Q}): aq \mbox{ is allowed in $Z$} \} \\
				U^-&:=\{(q, a)\in (\mathcal{A}_{Q}) \times (\mathcal{A}_{Z} \setminus \mathcal{A}_Q): qa \mbox{ is allowed in $Z$} \}.
			\end{align*}
			After recoding $Z, P, Q$ such that  (I) (II) (III) above hold, then $Z$ admits an $f$-blowup if and only if  
			there exist functions $c^+:U^+  \to P$ and $c^-:U^-  \to P$ such that 
			\begin{enumerate}
				\item $c^+(a,q) \in f^{-1}(\bar{q})$ for every $(a,q) \in U^+$, where $\bar{q} \in Q$ is the unique element satisfying $\bar{q}_0=q$.
				\item $c^-(q,a) \in f^{-1}(\bar{q})$ for every $(q,a) \in U^-$.
				\item If  $w = (w_0,\ldots, w_{n+1}) \in (\mathcal{A}_{Z})^{n+2}$ is an admissible word for $Z$ such that $w_0,w_{n+1} \in  (\mathcal{A}_{Z} \setminus \mathcal{A}_Q)$ and $(w_1,\ldots,w_n) \in  \mathcal{A}_{Q}^n$ is an admissible word for $Q$, then 
				\begin{equation} \label{simpler-parity-condition}
					\sigma^n(c^+(w_0,w_1))=c^-(w_n,w_{n+1}).
				\end{equation}
			\end{enumerate}    
			\noindent Moreover, if $\hat{Z}$ is irreducible, then  for any orbit $\mathcal{O}$ in $P$, there exist $\bar{p},\bar{p}'\in \mathcal{O}$, $q, q'\in \mathcal{A}_Q$ and $a, a'\in (\mathcal{A}_Z\setminus \mathcal{A}_Q)$ such that $c^+(a,q)=\bar{p}$ and $c^-(q',a')=\bar{p}'$.
		\end{corollary}

	The functions $c^+$ and $c^-$ satisfying Proposition \ref{characterization-blowup} are not unique.  A particular choice of $c^+$ and $c^-$ is buried in the necessity proof (Equation (\ref{d+-d-_defnintino})) of Proposition \ref{characterization-blowup}. In the language of Corollary \ref{simpler-chara-blowup}, that particular choice (call them $d^+$ and $d^-$) can be defined as the following.
	
	\begin{definition} \label{def_d+_d-}
		Under the same assumptions in Corollary \ref{simpler-chara-blowup}, we define, for each $(a,q)\in U^+$, $d^+(a,q)$ to be the unique periodic point $\bar{p}\in P$ such that $\pi^{-1}(z)_{0}=\bar{p}$ for all $z\in Z$ with $z_{[-1,0]}=aq$. The function $d^-$ is defined similarly.
	\end{definition} 
		
	We need another definition.
		\begin{definition} \label{def_standard_form}
		Let $\pi: \hat{Z}\to Z$ be an $f$-blowup, where $Z$ is a subshift and $f: P\to Q$ is a factor code from a finite subshift $P$ to another finite subshift $Q$. We say that $\pi: \hat{Z}\to Z$ is an $f$-blowup in {\em standard form} if the following holds:
		\begin{enumerate}
			\item[(I)] $f:P \to Q$ is a one block code, that is, there exists $F:\mathcal{A}_{P} \to \mathcal{A}_Q$ such that $f(p)_0=F(p_0)$ for every $p \in P$.
			\item[(II)] The projections $z \mapsto z_0$ and $\hat z \to \hat z_0$ from $Q$ to $\mathcal{A}_Q$ and from $P$ to $\mathcal{A}_P$ are bijective.
			\item[(III)] $\euscript{B}(Q)=\euscript{B}(Z) \cap \cup_{n} (\mathcal{A}_Q)^n$.
			\item[(IV)] $\pi$ is a one-block code
			\item[(V)] $\pi(\hat z)_0 = F(\hat z_0)$ whenever $\hat z_0 \in \mathcal{A}_P$.
			\item[(VI)] $\pi(\hat z)_0 = \hat z_0$ whenever $\hat z_0 \not \in \mathcal{A}_P$.
		\end{enumerate}
		\end{definition}
		
		Note from Equation (\ref{1-block-standard-form}) that the blowup $\phi:\hat{Z}\to Z$ constructed in the proof of sufficiency part of Proposition \ref{characterization-blowup} is of standard form. Thus, we have the following proposition.
		
		
		\begin{proposition} \label{blowup-and-standard-form}
			Let $Z$ be a subshift. Then 
			for any $f$-blowup $\phi: \hat{Z}\to Z$, there is an $f$-blowup in standard form with the same functions $d^+, d^-$ as $\phi$.
		\end{proposition}

		
		The following example shows that a sofic shift $Z$ can admit {\em non-conjugate $f$-blowups} for a fixed factor map $f: P\to Q$. That is, $Z$ may have two $f$-blowups $\pi_1: W_1 \to Z$ and $\pi_2: W_2\to Z$ for which there is no conjugacy $\tau: W_1\to W_2$ satisfying $\pi_2\circ \tau = \pi_1$.
		
		\begin{example} \label{blowups-may-not-conjugate}
			Let $Z, W_1, W_2$ be sofic shifts described by the the three vertex-labeled graphs in Figure~\ref{2blowup-notunique}, respectively. Let $\pi_1: W_1\to Z$ and $\pi_2: W_2\to Z$ both be factor maps induced by the $1$-block map that maps $0, 1$ to $a$ (and acts as an identity map on $b$ and $c$). Also let  $f:P\to Q$ where $P:=\{(01)^\infty, (10)^\infty\}$ and $Q:= \{a^\infty\}$. Then, $\pi_1: W_1\to Z$ and $\pi_2: W_2\to Z$ are both $f$-blowups of $Z$. We claim that there is no conjugacy $\tau: W_1\to W_2$ such that $\pi_2\circ \tau=\pi_1$. 
			
			To see this, assume to the contrary that such a $\tau$ exists. We also assume without loss of generality that $\tau((01)^\infty.(01)^\infty)=(0 1)^\infty.(01)^\infty.$ 
			Define $\psi_1: Z\setminus\{a^\infty\}\to W_1\setminus\{(01)^\infty, (10)^\infty\}$ by $\psi_1(u)=\pi_1^{-1}(u)$ and define $\psi_2: Z\setminus \{a^\infty\}\to W_2\setminus \{(0 1)^\infty, (1 0)^\infty\}$ by $\psi_2(u)=\pi_2^{-1}(u)$. Then restricted to $Z\setminus \{a^\infty\}$, $\tau\circ \psi_1 =\psi_2$. 
			For each $n\geq 1$, consider
			$$
			y^{(n)}:= \cdots c a^{2n}.a^{2n}b\cdots
			$$
			where the dots are allowed blocks in $Z$ such that $y^{(n)}\in Z$.
			Then, 
			\begin{align*}
				\psi_1(y^{(n)})=\cdots c(01)^n.(01)^nb\cdots \quad \mbox{and} \quad
				\psi_2(y^{(n)})=\cdots c(10)^n.(1 0)^nb\cdots.
			\end{align*}
			Note that $\psi_2(y^{(n)})$ converges to $(10)^\infty. (10)^\infty$ by the last equality above; on the other hand, since $\tau\circ \psi_1=\psi_2$, we have $\psi_2(y^{(n)})=\tau\circ \psi_1(y^{(n)})=\tau(\cdots c(01)^n.(01)^nb\cdots)$, which converges to $(01)^\infty.(0 1)^\infty$, a contradiction. 
			
			
			\begin{figure}[H]
				\begin{center}
					\begin{tikzpicture}[scale=0.35]  
						\draw [opacity=0] (0,0) grid (40,15);
						\node [circle, draw, thick] (a1) at (6,14) {$b$};
						\node [circle, draw, thick] (a2) at (9,10.5) {$a$};
						\node [circle, draw, thick] (a3) at (9,5.5) {$a$};
						\node [circle, draw, thick] (a4) at (6,2) {$c$};
						\node [circle, draw, thick] (a5) at (3,5.5) {$a$};
						\node [circle, draw, thick] (a6) at (3,10.5) {$a$};
						\draw [-stealth, black, thick] (a1) to [out=320, in=120] (a2);
						\draw [-stealth, black, thick] (a2) to [out=300, in=60] (a3);
						\draw [-stealth, black, thick] (a3) to [out=120, in=240] (a2);
						\draw [-stealth, black, thick] (a3) to [out=240, in=40] (a4);
						\draw [-stealth, black, thick] (a4) to [out=140, in=300] (a5);
						\draw [-stealth, black, thick] (a5) to [out=120, in=240] (a6);
						\draw [-stealth, black, thick] (a6) to [out=300, in=60] (a5);
						\draw [-stealth, black, thick] (a6) to [out=60, in=220] (a1);
						\node at (6, -1) [coordinate, draw, fill=black, label=below: Presentation of $Z$] {};
						\node [circle, draw, thick] (a1) at (21,14) {$b$};
						\node [circle, draw, thick] (a2) at (24,10.5) {$0$};
						\node [circle, draw, thick] (a3) at (24,5.5) {$1$};
						\node [circle, draw, thick] (a4) at (21,2) {$c$};
						\node [circle, draw, thick] (a5) at (18,5.5) {$0$};
						\node [circle, draw, thick] (a6) at (18,10.5) {$1$};
						\draw [-stealth, black, thick] (a1) to [out=320, in=120] (a2);
						\draw [-stealth, black, thick] (a2) to [out=300, in=60] (a3);
						\draw [-stealth, black, thick] (a3) to [out=120, in=240] (a2);
						\draw [-stealth, black, thick] (a3) to [out=240, in=40] (a4);
						\draw [-stealth, black, thick] (a4) to [out=140, in=300] (a5);
						\draw [-stealth, black, thick] (a5) to [out=120, in=240] (a6);
						\draw [-stealth, black, thick] (a6) to [out=300, in=60] (a5);
						\draw [-stealth, black, thick] (a6) to [out=60, in=220] (a1);
						\node at (21, -1) [coordinate, draw, fill=black, label=below: Presentation of $W_1$] {};
						\node [circle, draw, thick] (a1) at (33,14) {$b$};
						\node [circle, draw, thick] (a2) at (36,10.5) {$0$};
						\node [circle, draw, thick] (a3) at (36,5.5) {$1$};
						\node [circle, draw, thick] (a4) at (33,2) {$c$};
						\node [circle, draw, thick] (a5) at (30,5.5) {$1$};
						\node [circle, draw, thick] (a6) at (30,10.5) {$0$};
						\draw [-stealth, black, thick] (a1) to [out=320, in=120] (a2);
						\draw [-stealth, black, thick] (a2) to [out=300, in=60] (a3);
						\draw [-stealth, black, thick] (a3) to [out=120, in=240] (a2);
						\draw [-stealth, black, thick] (a3) to [out=240, in=40] (a4);
						\draw [-stealth, black, thick] (a4) to [out=140, in=300] (a5);
						\draw [-stealth, black, thick] (a5) to [out=120, in=240] (a6);
						\draw [-stealth, black, thick] (a6) to [out=300, in=60] (a5);
						\draw [-stealth, black, thick] (a6) to [out=60, in=220] (a1);
						\node at (33, -1) [coordinate, draw, fill=black, label=below: Presentation of $W_2$] {};
					\end{tikzpicture}
				\end{center}
				\caption{An example where the $f$-blowup is not unique up to conjugacy.}
				\label{2blowup-notunique}
			\end{figure}
		\end{example}

		Though Example \ref{blowups-may-not-conjugate} suggests that a subshift may have many non-conjugate blowups, the following proposition shows that any $f$-blowup must be conjugate to an $f$-blowup in standard form.

		
		\begin{proposition} \label{2blowup_conj}
			Let $\pi: X\to Z$ and $\phi:\hat{Z}\to Z$ both be $f$-blowups of $Z$. If $\pi$ and $\phi$ have the same functions $d^+$ and $d^-$ as defined in Definition \ref{def_d+_d-}, then there is a conjugacy $\tau: X\to \hat{Z}$ such that the following diagram commutes
			\begin{equation} 
				\xymatrix{
					X \ar[dr]_\pi \ar[rr]^\tau & & \hat{Z} \ar[dl]^\phi  \notag\\
					&	Z & }.
			\end{equation}
			In particular, any $f$-blowup is conjugate to an $f$-blowup in standard form in the sense of the diagram above. 
		\end{proposition}

		\begin{proof}
			
			By higher-block recoding, we may assume that conditions 
			(I) (II) and (III) of Definition \ref{def_standard_form}
			are satisfied.
			
			Let $f: P\to Q$ where $P$ and $Q$ are finite subshifts with $Q\subseteq Z$. We denote $P_1:= \pi^{-1}(Q)$ and $P_2:= \phi^{-1}(Q)$. Note that $P_1, P_2$ are both copies of $P$. 
			
			Define $\tau: X\to \hat{Z}$ as follows: 
			\begin{itemize}
				\item[-] Case 1: If $x\in X\setminus P_1$, then let 
				\begin{equation} \label{2025-05-15-a}
					\tau(x):= \phi^{-1}\circ \pi(x)
				\end{equation}
				Since $\phi^{-1}$ is continuous on $Z\setminus Q$, $\tau$ is continuous on $X\setminus P_1$. 
				\item[-] Case 2: Let $x\in P_1$ and denote $\bar{q}:= \pi(x)$. Note that $\bar{q}\in Q$.
				
				If $x$ is an isolated point in $X$, then choose any {$\hat{z} \in \phi^{-1}(\{\bar{q}\})$} 
				with $\per(x)=\per(\hat{z})$ and define 
				\begin{equation*}
					\tau(\sigma^j(x)):=\sigma^j(\hat{z}) \quad \mbox{for each $j\in\{0, 1, \cdots, \per(x)-1\}$};
				\end{equation*}
				If $x$ is not an isolated point, then choose an arbitrary sequence $\{x^{(n)}\}_{n=1}^\infty\subset X\setminus P_1$ such that $x^{(n)}$ converges to $x$. For each $n$, $\tau(x^{(n)})$ is defined by Equation (\ref{2025-05-15-a}) and is in $\hat{Z}$. By passing to a subsequence, we may assume that $\tau(x^{(n)})$ converges to a point $\hat{z}\in \hat{Z}$. By the continuity of $\phi$, $\phi(\tau(x^{(n)}))$ converges to $\phi(\hat{z})$. But note that $\phi(\tau(x^{(n)}))=\pi(x^{(n)})$ and it converges to $\bar{q}$ by the continuity of $\pi$. Thus, we have $\phi(\hat{z})=\bar{q}\in Q$ and therefore $\hat{z}\in P_2$.
				In this case, we define 
				\begin{equation} \label{2025-05-15-b}
					\tau(\sigma^j(x)):= \sigma^j (\hat{z}) \quad \mbox{for each $j\in \{0,1\cdots , \per(x)-1\}.$} \notag
				\end{equation}
			\end{itemize}
			
			
			
			We first show that $\tau$ is well-defined by showing that for each $j$, the value of $\tau(\sigma^j(x))$ defined in Case 2 above only depends on $x$ (and is independent of the choice of $\{x^{(n)}\}_{n=0}^\infty$). It suffices to prove this for $j=0$, i.e., $\hat{z}$ only depends on $x$. Moreover, by Condition (II) of Definition \ref{def_standard_form}, it suffices to show $\hat{z}_0=x_0$ (here we are identifying $P_1$ and $P_2$ to avoid extra notations).  
			
			To this end, 
			denote $z^{(n)}:= \pi(x^{(n)})$.
			Let $U^+, U^-$ be defined as in Corollary  \ref{simpler-chara-blowup}. Since  $z^{(n)}$ converges to $\bar{q}$, there exist $R\in \mathbb{N}^+$ such that $ z^{(n)}_0=\bar{q}_0$ for all $n\geq R$. Moreover, for each $n\geq R$, since $z^{(n)}\neq \bar{q}$, either \begin{enumerate}
				\item[(A)] there exist $m<0$ such that $z^{(n)}_{m}\neq \bar{q}_{m}$ and $z^{(n)}_{[m+1, 0]}=\bar{q}_{[m+1,0]}$; or
				\item[(B)] there exist $M>0$ such that $z^{(n)}_{M}\neq \bar{q}_{M}$ and $z^{(n)}_{[0, M-1]}=\bar{q}_{[0,M-1]}$;
			\end{enumerate}
			In case (A), by Condition (III) of Definition \ref{def_standard_form}, we have $(z^{(n)}_m, z^{(n)}_{m+1})\in U^+$. Since $d^+$ is associated with $\pi$, we infer from the definition of $d^+$ that $d^+(z^{(n)}_m, z^{(n)}_{m+1})$ is the unique point in $P_1$ with $(d^+(z^{(n)}_m, z^{(n)}_{m+1}))_0=(\pi^{-1}(z^{(n)}))_0$, i.e., 
			\begin{equation} \label{conjugacy-prop-x^n-0coor}
				(d^+(z^{(n)}_m, z^{(n)}_{m+1}))_0=x^{(n)}_0;
			\end{equation}
			Similarly, since $\phi$ has the same $d^+$ and $\tau(x^{(n)})=\phi^{-1}(\pi(x^{(n)}))=\phi^{-1}(z^{(n)})$, we have 
			\begin{equation} \label{conjugacy-prop-tau(x)^n-0coor}
				(d^+(z^{(n)}_m, z^{(n)}_{m+1}))_0=(\tau(x^{(n)}))_0.
			\end{equation}
		Now, from (\ref{conjugacy-prop-x^n-0coor}) and (\ref{conjugacy-prop-tau(x)^n-0coor}), we have $x^{(n)}_0=(\tau(x^{(n)}))_0$ (for all $n\geq R$). Thus,  $x_0=\hat{z}_0
		$ since $x^{(n)}$ converges to $x$ and $\tau(x^{(n)})$ converges to $\hat{z}$.
		
		In case (B), a similar argument (but using $d^-$ rather than $d^+$) as in Case (A) again gives $x_0=\hat{z}_0$. 
		Thus, $\hat{z}$ is independent of the choice of $x^{(n)}$ and therefore $\tau$ is well-defined.
		
		It remains to show $\tau$ is a conjugacy from $X$ to $\hat{Z}$.
		
		First note that $\tau$ is continuous on $X$ by construction. It can also be readily checked that $\tau$ is also shift-commuting. 
		
		To see $\tau$ is bijective, first note that $\tau$ is bijective on $X\setminus P_1$ since $\pi$ is bijective there and $\phi^{-1}$ is bijective on $\pi(X\setminus P_1)$. Now on $P_1$, it has been proven that for any $x\in P_1$, $\tau$ maps $x$ to a point $\hat{z}$ in $\hat{Z}$ with $x_0=\hat{z}_0$. Since the projection from $P$ to $\mathcal{A}_P$ is bijective (by Condition (II)) and $P_1, P_2$ are copies of $P$, we know that $\tau$ is the identity map and is therefore bijective on $P_1$ (here again we are identifying $P_1$ and $P_2$). Hence, $\tau$ is bijective on the entire $X$.
		
		Finally, as a continuous, bijective map between two compact sets, $\tau$ must be a homeomorphism, i.e., a conjugacy.

		The ``in particular" part follows from Proposition \ref{blowup-and-standard-form}.
	\end{proof}

	{\begin{proposition} \label{individual-blowup}
			Let $Z$ be an irreducible sofic shift and $Q_1, Q_2$ be two {disjoint} finite subshifts in $Z$. Let $P_1, P_2$ be two finite subshifts and $f_i: P_i\to Q_i$ be a factor code for $i=1,2$. Denote $f: P_1\cup P_2 \to Q_1\cup Q_2$. Then, {up to conjugacy}, $Z$ has an $f$-blowup at $Q_1\cup Q_2$ if and only if $Z$ has an $f_1$-blowup at $Q_1$ and has an $f_2$-blowup at $Q_2$.
		\end{proposition}
	}

	\begin{proof}
		We first prove the sufficiency. Suppose $\phi_X: X\to Z$ is an $f_1$-blowup of $Z$ at $Q_1$ and $\phi_W: W\to Z$ is an $f_2$-blowup of $Z$ at $Q_2$. Let $(\hat{Z}, \psi_X, \psi_W)$ be a fiber product of $(\phi_X, \phi_W)$ such that the following diagram commutes
		
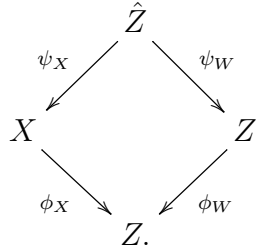
\begin{equation} 
			\xymatrix{
				& \hat{Z} \ar[dl]_{\psi_X} \ar[dr]^{\psi_W}& \\
				X \ar[dr]_{\phi_X}  & & Z \ar[dl]^{\phi_W}  \notag\\
				&	Z .& }
		\end{equation}
		Let $\pi: \hat{Z}\to Z$ be given by $\pi:= \phi_X\circ\psi_X$. We claim that $\pi$ is conjugate an $f$-blowup of $Z$ at $Q_1\cup Q_2$. 
		The fact that $\hat{Z}$ is a subshift and that $\pi$ is a factor code are direct consequences of the fiber product.
		Since $\vert \phi_X^{-1} (z) \vert=1$ for $z\notin Q_1$ and $\vert \phi_W^{-1} (z) \vert=1$ for $z\notin Q_2$, by the definition of the fiber product, any point in $Z$ that is not in $Q_1\cup Q_2$ has a unique $\pi$-preimage.
		Moreover, for any $\bar{q}\in Q_1$, since $Q_1$ and $Q_2$ are disjoint, it is easy to verify that $x$ is a $\phi_X$ preimage of $\bar{q}$ if and only if $(x, \phi_{W}^{-1}(\bar{q}))$ is a $\pi$-preimage of $\bar{q}$. 
		Since the $\phi_X$-preimage of $Q_1$ is $P_1$, the $\pi$-preimage of $Q_1$ is $\hat{P_1}$, where 
		$$
		\hat{P_1}=\cup_{\bar{p}\in P_1}\{(\sigma^{i}(\bar{p}), \sigma^{i}(\phi_W^{-1}\circ f(\bar{p}))): 0\leq i\leq \per(\bar{p}) \}.
		$$ 
		Similarly,  the $\pi$-preimage of $Q_2$ is $\hat{P_2}$, where 
		$$
		\hat{P_2}=\cup_{\bar{p}\in P_2}\{(\sigma^{i}(\phi_X^{-1}\circ f(\bar{p})), \sigma^i (\bar{p})): 0\leq i\leq \per(\bar{p}) \}.
		$$ 
		{Note that $\pi: \hat{Z}\to Z$ is conjugate to an $f$-blowup of $Z$ at $Q_1\cup Q_2$ since $\hat{P_1}$ is conguate to $P_1$ and $\hat{P_2}$ is conjugate to $P_2$.}
		
		We now prove the necessity. Let $\pi: \hat{Z}\to Z$ be an $f$-blowup of $Z$ at $Q_1\cup Q_2$. It suffices to show that $Z$ has an $f_1$-blowup at $Q_1$.
		
		By recoding, we may assume that $\pi$ is a 1-block code. 
		Let $g$ be the 1-block code on $\hat{Z}$ defined by $g(c) = \pi(c)$ for all $c$ except for $c$ being a symbol in some $\bar{p}\in P_1$, in which case we define $g(c)=c$. Let $X:=g(\hat{Z})$. Let $h:X \to Z$ be defined by $h(c) = c$ except that if $c$ appears in some $\bar{p}\in P_1$, then $h(c)=\pi(c)$. Then $h$ is an $f_1$-blowup of $Z$ at $Q_1$.
	\end{proof}

	\begin{proposition}	
		\label{right closing extension}
		Let $\hat{Z}$ be an irreducible shift space and $\pi:\hat{Z} \to Z$ be a right-closing factor map. If $Z$ is sofic, then so is  $\hat{Z}$.  In particular, an irreducible blowup of an irreducible sofic shift is sofic.
	\end{proposition}

	The proof of Proposition~\ref{right closing extension} is a modification of an analogous result for SFTs, namely~\cite[Proposition 4.12 (alternative proof, due to Bruce Kitchens)]{BMT} as follows.
	
	\begin{proof} Let $\phi:X \to Z$ be the minimal right-resolving presentation of $Z$. 
		Let $\hat{z}$ be a point in $\hat{Z}$ with a dense forward orbit. Then $z:=\pi(\hat{z})$ has a dense forward orbit in $Z$. Let $x\in X$ be any $\phi$-preimage of $z$. Since $\phi:X\to Y$ is the minimal right-resolving presentation which is in particular finite-to-one, we infer from the proof of \cite[Lemma 9.1.13]{LM} (by replacing ``$x$ is doubly transitive" with ``$x$ has a dense forward orbit") that $x$  has a dense forward orbit in $X$. 
	
	Let $(\psi:W \to X, \eta: W \to \hat{Z})$ be the fiber product of $(\phi,\pi)$. Then $w:=(x, \hat{z}) \in W$.  Let $W'\subset W$ be the 
	$\omega$-limit set of $w$, i.e.,  the set of limit points of the forward orbit of $w$. Then, by \cite[Lemma 2.1]{HSZ} (see also \cite{Bow}), $W'$ is chain transitive, i.e., for any $\epsilon>0$ and any $x,y\in W'$, there is a sequence $\{z^{(k)}\}_{k=1}^n$ in $W'$ such that $z^{(1)}=x$, $z^{(n)}=y$ and $d(\sigma(z^{(k)}), z^{(k+1)})<\epsilon$ for all $1\leq k\leq n-1$ (such a sequence $\{z^{(k)}\}_{k=1}^n$ is called an $\epsilon$-chain in $W'$ from $x$ to $y$). Since the forward orbit of $x$ is dense in $X$, we have  $\psi(W') = X$; since the forward orbit of $\hat{z}$ is dense in $\hat{Z}$, we have $\eta(W') = \hat{Z}$.

    We claim that $W'$ is an (irreducible) SFT. It will
    then follow that $\hat{Z}$ is sofic.
    
    To prove the claim, first observe that since $\pi$ is right-closing, by Proposition \ref{fibre}, so is $\psi$. By recoding, up to conjugacy we can assume that 
	$\psi$ is actually right-resolving and $X$ is a 1-step SFT. Let $A$ be the square 0-1 matrix indexed by the alphabet of $W'$ and defined by 
	$$A_{ij} = 1 \quad \mbox{iff}  \quad ij \in \euscript{B}(W').$$  
    
We show that $A$ is irreducible. Indeed, let $i, j$ be two arbitrary states of $A$ and let $x, y$ be two points in $W'$ such that $x_0=i$ and $y_0=j$. Recall that for all $\epsilon>0$, there exist an $n$ and an $\epsilon$-chain $\{z^{(k)}\}_{k=1}^n$ in $W'$ from $x$ to $y$. Choose $\epsilon$ so small that  for all $0\leq k<n-1$,  $\sigma(z^{(k)})_0=z^{(k+1)}_0$, or equivalently, $z^{(k)}_1=z^{(k+1)}_0$. Denote $a_k:= z^{(k)}_0$ for all $0\leq k\leq n-1$. Now for each $k$, $a_k a_{k+1}=z^{(k)}_0 z^{(k+1)}_0=z^{(k)}_0 z^{(k)}_1$. Since $z^{(k)}_0 z^{(k)}_1\in \mathcal{B}(W')$, we have $A_{a_k a_{k+1}}=1$ for all $1\leq k\leq n-1$. Thus, $A$ is irreducible.
    
	Let $X_A$ denote the 1-step SFT defined by $A$.
	Then $W' \subseteq X_A$. Since  $W' \subseteq X_A$ and  $X$ is a 1-step SFT, $\psi$ extends to a right-resolving map from $X_A$ into $X$.  But since $\psi(W') = X$, we must have 
	$\psi(X_A) = X.$ So, we have $\psi(X_A) = X = \psi(W')$. 
   Moreover, since $A$ is an irreducible matrix, $X_A$ is an irreducible SFT. 

    Since $X_A$ is irreducible, it is entropy minimal. Since $\psi$ is right-resolving, 
	$h(X_A) = h(X) = h(W')$. It follows that $W' = X_A$ is an irreducible SFT as desired.

	
	
	
	
	Finally, since the factor map from an irreducible blowup of an irreducible sofic shift onto the sofic shift is necessarily right-closing, it follows that every irreducible blowup of an irreducible sofic shift is itself sofic.
\end{proof}

The irreducibility assumption on the extension $\hat{Z}$ in Proposition \ref{right closing extension} cannot be removed as shown in the following example.

\begin{example}\label{example:non_sofic_2_to_1_extension}
	Let $Z$ be any SFT with positive entropy (for instance $Z=\{0,1\}^\mathbb{Z}$). Let $Y$  be a non-sofic subshift that embeds in $Z$ via $\iota:Y \to Z$. Let $\hat{Z}$ be the disjoint union of $Y$ and $Z$, and let $\pi:\hat{Z} \to Z$ be given by 
	\[
	\pi(\hat z)= \begin{cases}
		\hat z & \hat z \in Z\\
		\iota(\hat z) & \hat z \in Y.
	\end{cases}
	\]
	Then $\pi:\hat Z \to Z$ is a bi-closing factor map, but $\hat{Z}$ is non-sofic.
\end{example} 


\section{Decidablity of Conditions in Theorem \ref{embedding_into_near_markov}} \label{section:finite_decidablity}

In this section we show that  there is a finite procedure which decides,  given an irreducible near Markov shift $Y$ and an irreducible sofic shift $Z$ with $h(Z) < h(Y)$,
whether $Z$ embeds into $Y$ as a proper subshift. 



Let $Z$ be an irreducible sofic shift, $(G, L)$ be an irreducible presentation of $Z$ and $Q\subseteq Z$ be a finite subshift. We say that $(G, L)$ is {\em $Q$-adapted} if 
$Q$ is presented by $(H, L)$ where $H$ is the subgraph of $G$ consisting of all edges whose labels are in $\mathcal{A}_Q$.

\begin{lemma} \label{pre-SFT-is-SFT}
	Let $(G, L)$ be an irreducible presentation of an (irreducible) sofic shift $Z$, and $Q\subseteq Z$ be a finite subshift. Let $q$ denote the largest least period  of periodic points in $Q$. Then  $(G^{[2q+1]}, L^{[2q+1]})$ is a $Q^{[2q+1]}$-adapted presentation of $Z^{[2q+1]}$.
\end{lemma}
\begin{proof}
	By the argument in~\cite[Lemma 10.1,6]{LM}, one sees that for any finite subshift Q consisting of periodic points of least period at most $q$, any word of length $2q$ occurring in $Q$ is contained in exactly one orbit of $Q$. Thus, in the labeled subgraph $(H,L^{[2q+1]})$ of $(G^{[2q+1]}, L^{[2q+1]})$ consisting of all edges labeled by words of length $2q+1$ in $Q$, for any path $ee'$ of length 2, $L^{[2q+1]}(e)$
	and $L^{[2q+1]}(e')$ are in the same orbit of $Q$.
	It follows that 
	$(G^{[2q+1]}, L^{[2q+1]})$ is $Q^{[2q+1]}$-adapted. 
\end{proof}



Let $Z$ be an irreducible sofic subshift, $P$ and $Q$ be finite subshifts with $Q\subseteq Z$ and $f:P\to Q$ be a factor code. Let $(G,L)$ be an irreducible $Q$-adapted presentation of $Z$, and let $E_Q$ be the set of edges in the subgraph  representing $Q$. A new labeling $L^*: \mathcal{E}(G)\to (\mathcal{A_Z} \setminus \mathcal{A}_{Q})\cup \mathcal{A}_P$ is said to be  {\em $f$-adapted} if it satisfies the following conditions:
\begin{enumerate}
	\item[(i)] $L^*(e)=L(e)$ for all $e\in E(G)\setminus E_Q$.
	\item[(ii)] $L^*(e)\in P$ for all $e\in E_Q$ and for each $\bar{p}\in P$ there is an edge $e\in E_Q$ with $L^*(e)=\bar{p}$.
	\item[(iii)] $f(L^*(e))= L(e)$ for all $e\in E_Q$.
	\item[(iv)] {For each $g, g'\in E(G)\setminus E_Q$ and each $e, e'\in E_Q$ such that $t(g)=i(e)$ and $t(g')=i(e')$, if $L(g)=L(g')$ and $L(e)=L(e')$, then $L^*(e)=L^*(e')$.}
	\item[(v)] For each $h, h'\in E(G) \setminus  E_Q$ and each $e, e'\in E_Q$ such that $i(h)=t(e)$ and $i(h')=t(e')$, if $L(h)=L(h')$ and $L(e)=L(e')$, then $L^*(e)=L^*(e')$.
	\item[(vi)] If $ee'$ is a path of length $2$ in $G$ and $e, e'\in E_Q$, then $L^*(e')=\sigma (L^*(e))$.
\end{enumerate}

\begin{proposition}
	\label{relabel}
	Let $Z$ be an irreducible sofic shift, $P$ and $Q$ be finite subshifts with $Q\subseteq Z$, and $f:P\to Q$ a factor code. Suppose $Z, P, Q$ are recoded so that Conditions (I), (II), (III) of Definition \ref{def_standard_form} hold. Let $(G, L)$ be an irreducible $Q$-adapted presentation of $Z$.  Then $Z$ admits an irreducible $f$-blowup if and only if  there is an $f$-adapted relabeling $L^*$ of $(G, L)$. In fact, any irreducible $f$-blowup  $\pi:\hat{Z} \to Z$ in standard form is induced by an $f$-adapted relabeling $L^*$ of $(G, L)$ in the sense that the range of $L^*_\infty$ is $\hat{Z}$ and $\pi$ maps $L^*$-labelings to $L$-labelings.
\end{proposition}

\begin{proof}
	{\bf (Sufficiency):} Let $L^*$ be an $f$-adapted relabeling of $(G, L)$. Let $\hat{Z}$ be the irreducible sofic shift represented by $(G, L^*)$ and define $\pi: \hat{Z}\to Z$ by: $\pi(\hat{z})$ is the $L$-labeling of any bi-infinite path in $G$ whose $L^*$ labeling is $\hat{z}$.  Note from item (i) and (iii) that $\pi$ is well-defined. 
	
	
	We will prove that $\pi: \hat{Z}\to Z$ is an irreducible $f$-blowup. 
	
	Since $(G, L)$ is $Q$-adapted, there is a subgraph $H\subseteq G$ such that $Q$ is presented by $(H,L)$ (whose edges are denoted by $E_Q$). By item (ii), the $L^*$-labeling of elements in $E_Q$ exhaust all points in $P$. Therefore we have $\pi^{-1}(Q)=P$. 
	
	We also need to show any $z\in Z\setminus Q$ has a unique $\pi$-preimage. It suffices to show that if $\alpha, \beta$ are two bi-infinite walks in $G$ such that $L(\alpha)=L(\beta)=z$, then $L^*(\alpha)=L^*(\beta)$. To this end, observe that for each such pair $(\alpha, \beta)$, there is a partition of the integers into $\cup_i ([a_i, a_{i+1})$ such that for each odd $i$, $\alpha_{[a_i, a_{i+1})}$ (and therefore also $\beta_{[a_i, a_{i+1})}$) is in $H$, and for each even $i$, $\alpha_{[a_i, a_{i+1})}$ (and therefore also $\beta_{[a_i, a_{i+1})}$) is in $G\setminus H$. Here, if $a_i=-\infty$ (resp. $a_i=\infty$) for some $i$, then $[a_i, a_{i+1})$ (resp. $[a_{i-1}, a_i)$) is $(-\infty, a_{i+1})$ (resp. $[a_{i-1}, \infty)$)
	
	
	By item (i), for even $i$, $L^*(\alpha)_{[a_i, a_{i+1})}=L(\alpha)_{[a_i, a_{i+1})}$ and $L^*(\beta)_{[a_i, a_{i+1})}=L(\beta)_{[a_i, a_{i+1})}$. Recalling that $L(\alpha)=L(\beta)$, we have 
	\begin{equation} \label{even_i}
		L^*(\alpha)_{[a_i, a_{i+1})}=L^*(\beta)_{[a_i, a_{i+1})} \qquad \mbox{for  all even $i$}.
	\end{equation}
	
	Now for each odd $i$, $\alpha_{[a_i, a_{i+1})}$ and $\beta_{[a_i, a_{i+1})}$ are both paths in $H$. Moreover, since $L(\alpha)=L(\beta)\in Z\setminus Q$, we know that either $a_i\neq  -\infty$ or $a_{i+1}\neq \infty$. Now 
	\begin{itemize}
		\item    If $a_i\neq  -\infty$, then we have $\alpha_{a_{i}-1}\notin E_Q$ and $\beta_{a_{i}-1}\notin E_Q$. Moreover, observing that $t(\alpha_{a_{i}-1})=i(\alpha_{[a_i, a_{i+1})})$, $t(\beta_{a_{i}-1})=i(\beta_{[a_i, a_{i+1})})$ and $L(\alpha_{a_i-1})=L(\beta_{a_i-1})$, we infer from  item (iv) and (vi) that $L^*(\alpha_{[a_i, a_{i+1})})=L^*(\beta_{[a_i, a_{i+1})})$.
		\item    If $a_{i+1}\neq  +\infty$, we have $\alpha_{a_{i+1}}\notin E_Q$ and $ \beta_{a_{i+1}}\notin E_Q$. Moreover, observing that $t(\alpha_{[a_i, a_{i+1})})=i(\alpha_{a_{i+1}})$, $t(\beta_{[a_i, a_{i+1})})=i(\beta_{a_{i+1}})$ and $L(\alpha_{a_{i+1}})=L(\beta_{a_{i+1}})$, we infer from item (v) and (vi) that $L^*(\alpha_{[a_i, a_{i+1})})=L^*(\beta_{[a_i, a_{i+1})})$. 
	\end{itemize}

	Thus, we have 
	\begin{equation} \label{odd_i}
		L^*(\alpha)_{[a_i, a_{i+1})}=L^*(\beta)_{[a_i, a_{i+1})} \qquad \mbox{for  all odd $i$}.
	\end{equation}
	
	Combining Equation (\ref{even_i}) and (\ref{odd_i}), we have $L^*(\alpha)=L^*(\beta)$ and therefore $\pi$ is an $f$-blowup. Finally, note that $\hat{Z}$ is irreducible since $G$ is. Thus, the proof of the sufficiency part is completed.

	{\bf (Necessity):}  Suppose $Z$ admits an irreducible $f$-blowup. By Corollary \ref{simpler-chara-blowup}, there exist functions $c^+$ and $c^-$ such that Conditions (1)(2)(3) and the moreover part in Corollary  \ref{simpler-chara-blowup} holds.

	We now define a new edge labeling $L^*$ of $G$ by the following: 
	\begin{itemize}
		\item For each $e\in E(G)\setminus E_Q$, define $L^*(e)=L(e)$; 
		\item For each $e\in E_Q$ such that there is an $g\in E(G)\setminus E_Q$ with $t(g)=i(e)$, define $L^*(e):= c^+(L(g), L(e))$.
		\item Propagate the $L^*$-labeling to all edges in $E_Q$ in the sense that if $ee'$ is a path of length $2$ in $G$ and $e, e'$ are both in $E_Q$, then $L^*(e')=\sigma(L^*(e)).$
	\end{itemize}
	To show $L^*$ is well-defined, it suffices to show that given any $e'\in E_Q$, any two edges $g, g'$ in $E(G)\setminus E_Q$ that enter the subgraph $H$, and any path $\gamma$ in $H$ from $t(g)$ to $i(e')$, any path $\gamma'$ in $H$ from $t(g')$ to $i(e')$, then $L^*(e')$ defined by using $g, \gamma$ and using $g' \gamma'$ are the same, i.e., 
	\begin{equation} \label{dif-entrance-same-label}
		\sigma^{\vert \gamma \vert}(c^+(L(g), L(\gamma_1)))=\sigma^{\vert \gamma'\vert} (c^+(L(g'), L(\gamma'_1))).
	\end{equation}
	To see this, let $\xi$ be a path in $H$ from $t(e')$ to $i(h)$ for some edge $h\in E(G)\setminus E_Q$. Denote $k:= \vert \xi \vert$.
	
	Now consider the path $g\gamma e' \xi h$. Since $Z$ is presented by $(H,L)$, we have $L(\gamma e' \xi)$ is a subblock of some point in $Q$, and that $L(g)$ and $L(h)$ are not in $\mathcal{A}_Q$. Thus, we infer from Equation (\ref{simpler-parity-condition}) that 
	\begin{equation} \label{entrance_1}
		\sigma^{\vert \gamma \vert+k+1} (c^+(L(g), L(\gamma_1)))= c^-(L(\xi_{k}), L(h)).
	\end{equation}
	The same argument applied with $g\gamma e' \xi h$ replaced by $g' \gamma' e' \xi h$ gives 
	\begin{equation} \label{entrance_2}
		\sigma^{\vert \gamma' \vert+k+1} (c^+(L(g'), L(\gamma'_1)))= c^-(L(\xi_{k}), L(h)).
	\end{equation}
	Applying $\sigma^{-k-1}$ to both sides of (\ref{entrance_1}) and (\ref{entrance_2}), we obtain Equation (\ref{dif-entrance-same-label}). Therefore $L^*$ is well-defined.
	
	To see $L^*$ is $f$-adapted, it is clear from the definition of $L^*$ that it satisfies Condition (i) (ii) (iii) (iv) (vi). The proof that $L^*$ also satisfies Condition (v) is similar to the argument we used to prove Equation (\ref{dif-entrance-same-label}). Thus, $L^*$ is an $f$-adapted relabeling of $(G, L)$.
	
	Finally, for the ``in fact" part of the proposition, supppose we have an irreducible $f$-blowup in standand form. Then by the necessity proof, there is an $f$-adapted relabeling $L^*$ of $(G, L)$, which, by the sufficiency proof, gives us an irreducible $f$-blowup in standard form, and it is indeed the same as the $f$-blowup we started with.
\end{proof}

\begin{corollary} \label{consequences}
	\label{SFT-factor-through-blowup}
	Let $Z$ be an irreducible sofic shift, $P, Q$ be finite subshifts with $Q\subseteq Z$, and $f:P\to Q$ be a factor code. Then we have the following:
	\begin{enumerate}
		\item There is a finite procedure that decides if $Z$ admits an irreducible $f$-blowup.
		\item Any irreducible $f$-blowup of $Z$ is sofic.
		\item Let $\pi: \hat{Z}\to Z$  be an irreducible $f$-blowup of $Z$,  $\phi: X\to Z$ be an irreducible SFT cover of $Z$. Then $\phi$ factors through $\pi$, i.e. there is a factor code $\psi: X \to \hat{Z}$ such that the following diagram commutes: 
		\begin{equation} 
			\xymatrix{
				X \ar[dr]_\phi \ar[rr]^\psi & & \hat{Z} \ar[dl]^\pi  \notag\\
				&	Z & }.
		\end{equation}
	\end{enumerate}
\end{corollary}
\begin{proof}
	Given any presentation of $Z$,  
	the classical subset construction gives a right-resolving presentation of $Z$. Using that $Z$ is irreducible,  
	by passing to an irreducible component of full entropy, we can also assume that this presentation is irreducible.  
	Now, according to Lemma \ref{pre-SFT-is-SFT}, by passing to a higher block recoding we arrive at a irreducible, $Q$-adapted presentation $(G, L)$ of $Z$  that also satisfies Conditions (I) (II) and (III) of Definition \ref{def_standard_form}. By Proposition \ref{relabel}, $Z$ admits an irreducible $f$-blowup if and only if there is an $f$-adapted relabeling $L^*$ of $(G, L)$, which is clearly decidable by brute force. This proves (1).

	To see (2), note from Proposition \ref{relabel} that any irreducible $f$-blowup of $Z$ in standard form must be sofic since it is given by a labeled graph $(G, L^*)$. Moreover, Proposition \ref{2blowup_conj} shows that any irreducible $f$-blowup of $Z$ is conjugate to some irreducible $f$-blowup of $Z$ in standard form. It then follows that any irreducible $f$-blowup of $Z$ is sofic.
	
	We now prove (3). By Proposition~\ref{2blowup_conj}, we may assume that $\pi$ is in standard form. By a 
	standard recoding, there is an irreducible labeled graph $(G, L)$ and a conjugacy $\tau_1: X\to X_G$ such that the left part of Diagram (\ref{factor-through-diagram-proof}) commutes. Now by Lemma \ref{pre-SFT-is-SFT}, there exist conjugacies
	(by recoding to higher blocks) $\tau_2: X_G\to X_{G^{[2q+1]}}$ and $\tau_3: Z\to Z^{[2q+1]}$ such that $(G^{[2q+1]}, L^{[2q+1]})$ is $Q$-adapted and the middle part of Diagram (\ref{factor-through-diagram-proof}) commutes. By passing to further high-block recoding if necessary, we may also assume that conditions (I)(II)(III) of Definition \ref{def_standard_form} are satisfied. Finally, by the in fact part of Proposition \ref{relabel}, there is a relabeling map $L^*$ of $(G^{[2q+1]}, L^{[2q+1]})$, which induces the blowup in standard form $\pi':= \tau_3 \circ \pi \circ \tau_4: \hat{Z}^{[2q+1]}\to Z^{[2q+1]}$, and the right part of Diagram (\ref{factor-through-diagram-proof}) commutes. The desired result then follows by taking $\psi:=\tau_4\circ L^*_{\infty} \circ \tau_2 \circ \tau_1$. 
	\begin{equation}  \label{factor-through-diagram-proof}
		\xymatrix{
			X \ar[dr]_\phi \ar[r]^{\tau_1} &  X_G \ar[d]^{L_{\infty}}  \ar[r]^{\tau_2} & X_{G^{[2q+1]}} \ar[r]^{L^*_\infty} \ar[d]_{L^{[2q+1]}_\infty}& \hat{Z}^{[2q+1]} \ar[dl]^{\pi'} \ar[r]^{\tau_4} & \hat{Z}\\
			&	Z \ar[r]_{\tau_3} & Z^{[2q+1]} }.
	\end{equation}
\end{proof}

\begin{remark}
	Note that item (2) of Proposition \ref{consequences} gives a constructive proof of Proposition \ref{right closing extension} in the case of irreducible blowups.
\end{remark}

Propositions~\ref{p-periodic-finitely-decidable} and~\ref{q-periodic-characterization-by-disjointness} below give finite procedures for verifying 
$p$-periodicity.

\begin{proposition} \label{p-periodic-finitely-decidable}
	Let $Y$ be an irreducible sofic shift and $\pi: X_G\to Y$ be its minimal right-resolving cover. Let $\mathcal{L}$ be the labeling map defined on the edges of $G$ that induces $\pi$. Then $Y$ is $p$-periodic if and only if both of the following hold:
	\begin{enumerate}
		\item $X_G$ is $p$-periodic;
		\item Let $k =  |V_G|^2$. For all pairs $(e_{-k} \ldots e_k), (g_{-k} \ldots g_k)$ of edge walks in $G$ with $\mathcal{L}(e_{-k}\ldots e_k) =\mathcal{L}(g_{-k} \ldots g_k)$, there exists $0\leq i\leq p-1$ such that $[e_0]\subseteq X_i$ and $[g_0]\subseteq X_i$, where $\biguplus_{i=0}^{p-1}X_i $ is the unique (up to cyclic permutation) cyclic partition of $X_G$.
	\end{enumerate}
\end{proposition}

\begin{proof}
	\noindent {\bf (Sufficiency)} 
	Suppose conditions (1) and (2) hold. Since $X = X_G$ is a one-step SFT represented by the graph $G$, each $X_i$ is a union of cyclinder sets defined by edges in $G$. For each $i$, let $\mathcal{E}_i$ be the set of edges such that $X_i= \bigcup_{e\in \mathcal{E}_i} [e]$.
	Now for each $0\leq i\leq p-1$, let 
	$$Y_i:= \bigcup_{e_0\in \mathcal{E}_i}[\mathcal{L}(e_{-k}\cdots e_k)],$$ i.e., $Y_i$ is the union of cylinder sets defined by the labels of length $(2k+1)$-edge walks whose edge at the $0$-th coordinate is from $\mathcal{E}_i$. Then, $Y_0, Y_1, \cdots Y_{p-1}$ are disjoint because $\mathcal{E}_0, \mathcal{E}_1,\cdots \mathcal{E}_{p-1}$ are. Moreover, we have $Y_{i+1} = \sigma(Y_i) \mod p$ because for any edge walk $e_{-k}\cdots e_k$, $e_0\in \mathcal{E}_i$ if and only if $e_1\in \mathcal{E}_{i+1} \mod p$. Thus, $\biguplus_{i=1}^{p-1} Y_i$ is a cyclic partition of $Y$ and therefore $Y$ is $p$-periodic.
	
	\noindent{\bf (Necessity:)}
	Suppose $Y$ is $p$-periodic and let $\biguplus_{i=1}^{p-1} Y_i$ be the unique (up to cyclic permutation) cyclic partition of $Y$. 
	
	We first show $X$ is $p$-periodic. Indeed, Let  \begin{align} \label{decom_lift}
		X_i:= \pi^{-1}(Y_i) \qquad \mbox{for all $0\leq i\leq p-1$ }
	\end{align} Then $X_0, X_1,\cdots, X_{p-1}$ are disjoint closed sets and for each $i$, $X_{i+1} =\sigma (X_i) \bmod p$. Thus, $X$ is $p$-periodic.
	
	Note that $X_i$ must be defined by Equation (\ref{decom_lift}) up to cyclic permutation because the $p$-periodic decomposition of $X$ is unique up to cyclic permutation.
	
	We now turn to condition (2). 
	
	To this end, let $e_{-k}\cdots e_{k}$ and $g_{-k} \cdots g_{k}$ be two edges walks in $G$ such that $\mathcal{L}(e_{-k}\cdots e_{k})= \mathcal{L}(g_{-k}\cdots g_{k})$. By the Pigeon Hole principle, there exists $-k\leq j< l\leq k$ such that $(e_j, g_j)=(e_l, g_l)$. Note that $e_j\cdots e_l$ and $g_j\cdots g_l$ are two cycles in $G$ with the same label. 
	
	Let $\alpha:= (e_j\cdots e_l)^\infty$ and $\beta:=(g_j\cdots g_l)^\infty$, where $\alpha_{[j,l]}=e_j\cdots e_l$ and $\beta_{[j,l]}=g_j\cdots g_l$. Let $y$ denote the common $\mathcal{L}$-labeling of $\alpha$ and $\beta$. Since $y\in Y$, there exists $0\leq i\leq p-1$ such that $y\in Y_i$. Also, since $\alpha, \beta$ are both preimages of $y$, we have $\alpha\in X_i$ and $\beta\in X_i$. 
	
	Now consider $\sigma^l(\alpha)$. We know that $\sigma^l(\alpha)\in \sigma^{l}(X_i)$ and $\sigma^l(\alpha)_0=e_l$. Thus, $[e_l]\subseteq \sigma^{l}(X_i)$. 
	Let $\gamma$ be any bi-infinite walk in $G$ with $\gamma_{[-k, k]}=e_{-k}\cdots e_k$. Note that $\sigma^l(\gamma)_0=e_l$. Thus, $\sigma^l(\gamma)\in \sigma^{l}(X_i)$ because $[e_l]\subseteq \sigma^{l}(X_i)$. But this means $\gamma\in X_i$ and therefore $[e_0]\subseteq X_i$ because $\gamma_{0}=e_0$.
	
	By replacing $\sigma^l(\alpha)$ with $\sigma^{l}(\beta)$, a similar argument as above gives $[g_0]\subseteq X_i$, establishing condition (2).
\end{proof}




The following proposition reveals the relation between the canonical cyclic decomposition and the cyclic partition (the partition that witnesses $q$-periodicity) for an irreducible sofic shift. We refer the reader to Section \ref{section:preliminaries} for notations used below.


\begin{proposition} \label{q-periodic-characterization-by-disjointness}
	Let $Y$ be an irreducible sofic shift with global period $p$ and $$
	Y=D_0\sqcup D_1\cdots \sqcup D_{p-1}
	$$ 
	be the canonical cyclic decomposition of $Y$. Then $Y$ is $q$-periodic if and only if $q$ divides $p$ and $C_k^{(q)}$'s are disjoint, where  
	$$
	C^{(q)}_k:= \bigcup_{m=0}^{p/q-1} D_{k+qm} \quad \mbox{for all} \quad 0\leq k\leq q-1.
	$$ 
	Moreover, if $Y$ is $q$-periodic, then $Y=\biguplus_{k=0}^{q-1} C^{(q)}_k$ is the unique (up to permutation) $q$-periodic cyclic partition of $Y$. 
\end{proposition}

\begin{proof}
	The sufficiency is clear.

	Suppose $Y$ is $q$-periodic with a  partition
	$$
	Y=\biguplus_{j=0}^{q-1} Y_j
	$$ 
	where $Y_i$'s are clopen and disjoint, and 
	$\sigma(Y_i)=Y_{i+1} \bmod q$. 
	
	Since $(\sigma^p, D_0)$ is mixing, $(\sigma^{pq}, D_0)$ is mixing. So there is a point $y\in D_0$ such that the orbit $\{\sigma^{mpq}(y): m\in \mathbb{Z}\}$ is dense in $D_0$. For notational convenience, we suppose $y\in Y_0$ (otherwise permute $Y_i$'s cyclicly). Since $\sigma^{mpq}(Y_0)=Y_0$ for each $m$, we have $\{\sigma^{mpq}(y): m\in \mathbb{Z}\} \subset Y_0$ and therefore $D_0\subset Y_0$.
	Recall that
	$$\sigma(D_i)= D_{i+1} \bmod p \quad \mbox{and} \quad \sigma(Y_j)=Y_{j+1} \bmod q,$$ we have $q$ divides $p$ and  for all $0\leq k \leq q-1$ and all $0\leq m\leq p/q-1$,
	$
	D_{k+mq}\subset  Y_k,
	$
	which gives $C_{k}^{(q)}\subset Y_k$ for all $0\leq k\leq q-1$. Note that $\cup_{k=0}^{q-1} C_k^{(q)}=Y$, we indeed have $C_k^{(q)}=Y_k$ for all $0\leq k\leq q-1$. 
	As a consequence, $C_k^{(q)}$'s are disjoint because $Y_k$'s are.
\end{proof}

\begin{remark}
	According to the proof of Proposition \ref{q-periodic-characterization-by-disjointness}, any $q$-periodic partition $Y=\biguplus_{k=0}^{q-1} Y_k$ of $Y$ must satisfy $Y_k= C_k^{(q)}$, where $C_k^{(q)}$ are defined in Proposition \ref{q-periodic-characterization-by-disjointness}.  Thus, if an irreducible sofic shift is $q$-periodic, then the $q$-periodic witness must be unique (up to cyclic permutation). 
\end{remark}

\begin{remark}
	The reader can easily check that the disjointness of $C^{(q)}_k$ in Proposition \ref{q-periodic-characterization-by-disjointness} is decidable. Therefore, Proposition \ref{q-periodic-characterization-by-disjointness} gives another way of proving the decidability of $q$-periodicity of an irreducible sofic shift, which is an alternative to Proposition \ref{p-periodic-finitely-decidable}.   
\end{remark}

{
	\begin{theorem} \label{finite-deci-main-theorem}
		Let $Z$ be an irreducible sofic shift and $Y$ be an (irreducible) near Markov shift with $h(Z) < h(Y)$.  There is a finite procedure that determines whether $Z$ can be properly embedded into $Y$.\end{theorem}
	
	\begin{proof}
		Since $h(Z) < h(Y)$, Theorem \ref{embedding_into_near_markov} (part 2) gives necessary and sufficient conditions for properly embedding $Z$ into $Y$: namely the conditions 2(a)-(d), which in particular specifies the existence of a blow-up $\hat{Z}$ of $Z$. 
		Since we are assuming that $Z$ is irreducible, according to the ``Moreover'' of Theorem \ref{embedding_into_near_markov}, we may assume that $\hat{Z}$ is irreducible. 

		For each of the finitely many choices
		of factor codes $f_0: \hat{Z}_0 \to Z_0$ for which there are embeddings $\hat{\rho}_0, \rho_0$ that satisfy the commutative diagram (\ref{commdiag}) in Theorem \ref{embedding_into_near_markov}, Corollary \ref{consequences} (1) gives a finite procedure to decide whether there is  an irreducible $f_0$-blowup $\pi_Z: \hat{Z} \to Z$. In other words, it decides the existence of an $f_0$-blowup that satisfies conditions 2(a),2(b).


		{By the ``In fact'' part of Proposition~\ref{relabel}, for a given $f_0$, we obtain a list  of all irreducible $f_0$-blowups in standard form, we can decide which ones of these (if any) are $p$-periodic by Proposition~\ref{p-periodic-finitely-decidable}.  Since, every $f_0$-blowup is conjugate to one in standard form, 
			this handles condition 2(d). }

		Finally we turn to checking condition 2(c). By definition of blowup this is the same as checking
		$$q_{n}(\hat{Z}\setminus \hat{Z_0}) \le q_{n}(\hat{Y}  \setminus \hat{Y_0})$$ for all positive integers $n$.
		
		For each of the blowups $\hat{Z}$ that are $p$-periodic and each $n$, $q_n(\hat{Z}) = 0$ unless $p$ divides $n$.  Since $\hat{Y}$ is an irreducible SFT with  global period $p$,  $q_n(\hat{Y}) = 0$ unless $p$ divides $n$.
		
		Thus, to check condition 2(c), we must check 
		whether  
		$$q_{pk}(\hat{Z}\setminus \hat{Z_0}) \le q_{pk}(\hat{Y} \setminus \hat{Y_0})$$ for all positive integers $k$.
		
		
		In the following we write $h(\hat{Y}) = h(Y) = \log \lambda_Y$ and  $h(\hat{Z}) = h(Z) = \log \lambda_Z$, where, of course, $0 < \lambda_Z < \lambda_Y. $
		
		We claim that there is a lower bound:
		\begin{equation}
			\label{lower_bound}
			q_{pk}(\hat{Y}) \ge p(\lambda_Y)^{pk} - 
			\sum_{r \in R} a_r r^{pk}
		\end{equation}
		for some finite set $R$ of numbers $0 < r < \lambda_Y$ and each $a_r > 0$.
		Here, $R$ consists of $\sqrt{\lambda_Y}$ and the absolute values  of all nonzero eigenvalues $ < \lambda_Y$ of an adjacency matrix representing the SFT $\hat{Y}$. In fact, this is done in the computation on~\cite[pp. 350-351]{LM}. 
		
		We also claim that there is an upper bound:
		\begin{equation}
			\label{upper_bound}
			q_{pk}(\hat{Z}) \le  
			\sum_{s \in S} b_s s^{pk}
		\end{equation}
		for some finite set $S$ of numbers $0 < s 
		\le \lambda_Z$ and each $b_s > 0$.
		This can be done by using the eigenvalues in the adjacency matrices constructed in the proof of the rationality of the zeta function for sofic shifts~\cite[Theorem 6.4.6]{LM}.
		
		Combining these lower and upper bounds we obtain 
		\begin{equation}
			\label{combine}
			q_{pk}(\hat{Y}) - q_{pk}(\hat{Z})\ge p(\lambda_Y)^{kp} - 
			\sum_{r \in R} a_r r^{pk} -  \sum_{s \in S} b_s s^{pk}
		\end{equation}
		for all $k$. 
		
		Since $\lambda_Y > r$ for each $r \in R$ and 
		$\lambda_Y > \lambda_Z \ge s$ for each $s \in S$, it follows that for some $k = k_0$, the right hand side of (\ref{combine}) is strictly positive. We may assume that $k_0$ exceeds all the periods of points in $\hat{Y_0}$ and  $\hat{Z_0}$. An exercise in calculus shows that for $k \ge k_0$, the right-hand side is increasing. Thus, for all $k \ge k_0$, we have
		$$
		q_{kp}(\hat{Y}) - q_{kp}(\hat{Z}) > 0.  
		$$
		
		In order to check condition 2(c), it now remains only to check for each $k < k_0$:
		$$
		q_{kp}(\hat{Y}\setminus \hat{Y_0}) - q_{kp}(\hat{Z}\setminus \hat{Z_0}) \ge 0.
		$$
	\end{proof}
}

{\begin{remark}
		If $Y$ is  mixing, then the proof of Theorem \ref{finite-deci-main-theorem} is simpler: by \cite[Remark 5.3]{MMTW}, the global period of $Y$ is $1$. Since $\pi_Y: \hat{Y}\to Y$ is almost invertible, the global period of $\hat{Y}$ is also $1$ by \cite[Proposition 5.1]{MMTW}. Now Condition 2(d) in Theorem \ref{embedding_into_near_markov} trivially holds because any subshift is $1$-periodic.
	\end{remark}
}

{\em Acknowledgement:} We would like to acknowledge Klaus Thomsen for many helpful discussions on topics closely related to this paper.


\begin{thebibliography}{99}

    \bibitem[Bow]{Bow} R. Bowen, {\em $\omega$-limit sets for Axiom A diffeomorphisms}, Journal of Differential Equations {\bf 18} (1975), 333–339.
    
	\bibitem[B]{B} M. Boyle, {\em Lower entropy factors of sofic systems}, Ergod. Th. Dynam. Syst. {\bf 4} (1984), 541-557.
	
	\bibitem[BKM]{BKM} M. Boyle, B. Kitchens and B. Marcus, {\em A note on minimal covers for sofic systems}, Proceedings of the American Mathematical Society, 95, no. 3 (1985): 403–11.
	
	\bibitem[BK]{BK} M. Boyle and W. Krieger, {\em Almost Markov and Shift Equivalent Sofic Systems}, Springer Lecture Notes in Mathematics, no. 1348 (1988), 335-395. 
	
	\bibitem[BMT]{BMT} M. Boyle, B. Marcus and P. Trow, {\em Resolving maps and the dimension group for shifts of finite type}, Volume 70, Memoirs of the American Mathematical Society, 1987. 


    \bibitem[HSZ]{HSZ} M.W. Hirsch, H.L. Smith and XQ Zhao, {\em Chain Transitivity, Attractivity, and Strong Repellors for Semidynamical Systems}, Journal of Dynamics and Differential Equations {\bf 13} (2001), 107–131.
    
	\bibitem[Ki]{Ki} B. Kitchens, {\em Symbolic Dynamics: One-sided, Two-sided and countable stste}, Springer Universitext, 1988.  
	
	\bibitem[Kr1]{Kr1} W. Krieger, {\em On the subsystems of topological Markov chains}, Ergod. Th. Dynam. Sys. {\bf 2} (1982), 195-202.
	
	\bibitem[Kr2]{Kr2} W. Krieger, {\em On images of sofic systems},arXiv:1101.1750v2, 2018 (v1 in 2011). 
	
	\bibitem[Kr3]{Kr3} W. Krieger, {\em On the subsystems of certain sofic shifts}, arXiv:2507.02717v2, 2025.
	
	
	
	
	\bibitem[LM]{LM} D. Lind and B. Marcus, {\em An introduction to symbolic dynamics and coding}, second edition, Cambridge Mathematical Library, Cambridge University Press (2021).
	
	\bibitem[M]{M} B. Marcus, {\em Sofic systems and encoding data}, IEEE Trans. Inform. Theory, {\bf 31}, (1985), 366-377.
	
	\bibitem[MMTW]{MMTW} B. Marcus, T. Meyerovitch, K. Thomsen and C. Wu, {\em Factorizable Embeddings and the Period
		of an irreducible sofic shift}, Preprint, arXiv:2508.02554, 2025.
	
	\bibitem[T]{T} K. Thomsen, {\em On the structure of a sofic shift space}, Trans. Amer. Math. Soc. {\bf 356} (2004), 3557-3619.
	
	\bibitem[W]{W} B. Weiss, {\em Subshifts of finite type and sofic systems}, Monatshefte Mathematik,
	77 (1973), 462-474.
	
	
	
\end{thebibliography}
\end{document}